\documentclass[a4paper]{amsart} 

\usepackage{amsmath,amsthm,amssymb,amsfonts,mathrsfs,color,hyperref, mathtools,crop, graphicx, enumitem}
\usepackage[left=3cm, right = 3cm]{geometry}
\usepackage{todonotes}
\usepackage{color}
 
\theoremstyle{plain}
\begingroup
\newtheorem{theorem}{Theorem}[section]
\newtheorem{corollary}[theorem]{Corollary}

\newtheorem{lemma}[theorem]{Lemma}
\newtheorem{prop}[theorem]{Proposition}
\endgroup

\theoremstyle{definition}
\begingroup

\endgroup

\theoremstyle{remark}
\begingroup
\newtheorem{remark}[theorem]{Remark}
\endgroup 

\numberwithin{equation}{section}
\newcommand{\N}{\mathbb N} 
 
\newcommand{\R}{\mathbb R} 

\newcommand{\dist}{{\rm dist}}

\newcommand{\wto}{\rightharpoonup}

\newcommand{\E}{{\mathcal E}}
\newcommand{\W}{{\mathcal W}}

\renewcommand{\L}{{\mathcal L}}
\newcommand{\Hh}{{\mathcal H}}

\newcommand{\F}{{\mathcal F}}

\newcommand{\Ra} {\Rightarrow}
\newcommand{\sdist}{{\mathrm{sdist}}}
\renewcommand{\H}{{\mathcal H}}
\newcommand{\spt}{{\mathrm{spt}}}
\newcommand{\inv}{^{-1}}
\newcommand{\cc}{\Subset}
\newcommand{\embeds}{\xhookrightarrow{\quad}}
 
\newcommand{\eps}{\varepsilon}

\newcommand{\dx}{\mathrm{d}x}
\renewcommand{\d}{\,\mathrm{d}}
 
\definecolor{dkred}{rgb}{0.4, 0.0, 0.0}
\definecolor{dkblue}{rgb}{0.0, 0.0, 0.5}

 \newcommand{\stephan}[1]{\textcolor{black}{#1}}
 \newcommand{\mine}[1]{\textcolor{black}{#1}}
 \newcommand{\ton}[1]{\textcolor{black}{#1}}

\begin{document}

\title[Phase field models with topological constraint]{Phase field models for thin elastic structures with topological constraint}
\author{Patrick W.~Dondl}
\address{Patrick W.~Dondl\\Abteilung f\"ur Angewandte Mathematik\\Albert-Ludwigs-Universit\"at Freiburg\\Hermann-Herder-Str.~10\\ 79104 Freiburg i.~Br.\\Germany}
\email{patrick.dondl@mathematik.uni-freiburg.de}

\author{Antoine Lemenant}
\address{Antoine Lemenant\\Universit\'e Paris Diderot - Paris 7\\ 
U.F.R de Math\'ematiques\\
B{\^a}timent Sophie Germain\\
75205 Paris Cedex 13\\France}
\email{lemenant@ljll.univ-paris-diderot.fr}

\author{Stephan Wojtowytsch}
\address{Stephan Wojtowytsch\\Department of Mathematical Sciences\\Durham University\\Durham DH1\,1PT, United Kingdom}
\email{s.j.wojtowytsch@durham.ac.uk}

\date{\today}


\begin{abstract}
This article is concerned with the problem of minimising the Willmore energy in the class of \emph{connected} surfaces with prescribed area which are confined to a \mine{small} container. We propose a phase field approximation based on De Giorgi's diffuse Willmore functional to this variational problem. Our main contribution is a penalisation term which ensures connectedness in the sharp interface limit. The \mine{penalisation of disconnectedness} is based on a geodesic distance chosen to be small between two points that lie on the same connected component of the transition layer of the phase field. 

\mine{We prove that in two dimensions, sequences of phase fields with uniformly bounded diffuse Willmore energy and diffuse area converge uniformly to the zeros of a double-well potential away from the support of a limiting measure. In three dimensions, we show that they converge $\H^1$-almost everywhere on curves.} This enables us to show $\Gamma$-convergence to a sharp interface problem that only allows for connected structures. \mine{The results also imply Hausdorff convergence of the level sets in two dimensions and a similar result in three dimensions.}

We furthermore present numerical evidence of the effectiveness of our model. The implementation relies on a coupling of Dijkstra's algorithm in order to compute the topological penalty to a finite element approach for the Willmore term.
\end{abstract}

\maketitle
\tableofcontents

\section{Introduction}

In this article, we consider a diffuse interface approximation of a variational problem arising in the study of thin elastic structures. Our particular question is motivated by the problem of predicting the shape of certain biological objects, such as mitochondria, which consist of an elastic lipid bilayer and are confined by an additional outer mitochondrial membrane of significantly smaller surface area.

In order to describe the the locally optimal shape of thin elastic structures, we will rely on suitable bending energies. A well-known example of such a variational characterisation of biomembranes is given by the Helfrich functional~\cite{Helfrich:1973td, Canham:1970jp}

\begin{equation}
\E^\mathrm{H}(\Sigma) = \mine{ \int_\Sigma \chi_H\, (H-H_0)^2 + \chi_K\, K \d \H^2} \label{eq:Will_Hel}
\end{equation}
where $\Sigma$ denotes the two-dimensional membrane surface in $\R^3$ and $H$ and $K$ denote its mean and Gaussian curvatures. The parameters \mine{$\chi_H$, $\chi_K$} and $H_0$ are the bending moduli and the spontaneous curvature of the membrane \mine{and are, in the simplest case, constants}. \mine{The integral is} performed with respect to the two-dimensional Hausdorff measure $\Hh^2$. The special case when $H_0, \chi_{\mine{K}} = 0$ \mine{and $\chi_H=1$} is known as Willmore's energy\label{now labelled equation}
\begin{equation}
\W(\Sigma)= \int_\Sigma H^2\d \H^2.
\end{equation}
\mine{
While our results extend to certain Helfrich-type functionals of more general form (see section \ref{section conclusion}), the basic task we have set for ourselves is the following: {\it Minimise Willmore's energy in the class of connected $C^2$-surfaces which are embedded into a domain $\Omega$ such that the embedding has prescribed surface area $S>0$.} This problem does not necessarily admit a solution, but a minimiser exists in the varifold closure of these surfaces \cite{Dondl:2014vn}.
}

\mine{
In this version of the problem, the outer mitochondrial membrane is the rigid boundary $\partial\Omega$, but the results also apply if an elastic container is used, see section \ref{section conclusion}.
}

\mine{
We propose a phase field approach to this problem. For computations, this has some advantages to the naive approach of taking a gradient flow of Willmore's energy (or an area preserving version of it).
}

\mine{
First, Willmore flow is a fully non-linear degenerate parabolic PDE of fourth order on a moving surface, which makes both analytic and numerical treatment difficult. In contrast, our phase-field flow equations are fourth order quasi-linear parabolic with constant coefficients for the fourth derivatives. The price we have to pay for this relative luxury is, as usual, that we have to solve equations in one higher dimension.
}

\mine{
Second, Willmore flow is a fourth order PDE and as such does not admit a maximum principle. In particular, a surface originally confined to even a convex container could easily leave it along the gradient flow, whereas for us the confinement to $\Omega$ is simply given by the choice of boundary condition and the domain of the phase field.
}

\mine{
The main challenge for phase fields is to control the topology of a limiting surface (in some sense) from quantities of diffuse interfaces. In this paper, we give a suitable concept of connectedness that admits a rigorous variational statement as $\eps\to 0$ and can be implemented efficiently in a simulation. 
}

\mine{
Heuristically, we introduce a {\emph quantitative} notion of {\emph path-connectedness} for the interfacial region of a phase field. The disconnectedness for two points $x,y$ in the interface layer is measured by a geodesic distance (the infimum length of connecting paths) where walking in the interfacial region does not add any length, but leaving the interface adds a positive amount. The total disconnectedness is then computed as a double integral over the distances of interface points.
}

\mine{
This functional poses analytic challenges since it is not a priori clear how 'two-dimensional' the interfacial layer of a phase field is. Since the distance function uses the length of curves, away from a limiting surface we need the phase-fields to converge almost everywhere on curves, i.e.\ on objects of codimension $2$. Also, the structure of the interface has to be understood in a precise way in order to estimate the aforementioned double integral. More rigorous statements are given below. 
}

\mine{Our analysis builds on ideas in \cite{Roger:2006ta}.}
Along the proof of our Theorem, we obtain a number of useful technical results about the convergence of phase field approximations to the limit problem which appear to be new in this context. Namely, for sequences along of uniformly bounded \mine{diffuse energy}, \mine{in two dimensions} we prove convergence of the transition layers to the \mine{support of a limiting measure} in the Hausdorff distance and uniform convergence of the phase field to $\pm 1$ away from the limit curve. This is reminiscent of a result for minimisers of the Modica-Mortola functional among functions with prescribed integral \cite{Caffarelli:1995kh} and more generally \mine{for} stationary states \cite{Hutchinson:2000df}.

In three dimensions, we show that neither result is true, but prove \mine{$L^1$-convergence on curves and that $\spt(\mu)$ is contained in the Hausdorff limit of interfacial regions. This enables us to precisely control our topological energy term}. We also show that phase fields are uniformly $L^\infty$-bounded in terms of their Willmore energy, also in three dimensions.

The numerical implementation of our functional will be discussed further in a forth-coming article \cite{DW_numerical}. We use a variant of Dijkstra's algorithm similar to the one of \cite{benmansour:2010dm} to compute the geodesic distance function used in the topological term of our energy functional. The weight in the geodesic distance is chosen to be exactly zero along connected components of the transition layer, which implies that it only needs to be computed once per detected connected component. This makes the functional efficient from an implementation point of view.

The \mine{article} is organised as follows. \mine{In the remainder of this section we give some further context for our problem.} In section \ref{section review} we give a brief introduction to phase fields for Willmore's problem in general (sections \ref{section 2 background} {and \ref{section 2 discrepancy}}) and to our approach to connectedness for the transition layers (section \ref{section 2 connectedness}). Our main results are subsequently listed in section \ref{section 2 main}. 

Section \ref{section proofs} is entirely dedicated to the proofs of our main results. We directly proceed to show $\Gamma$-convergence of our functionals (section \ref{section 3 Gamma}). In sections \ref{section 3 auxiliary 1} and \ref{section 3 auxiliary 2}, we produce all the auxiliary estimates that will be needed in the later proofs. These are then used to show convergence of the phase fields, Hausdorff-convergence of the transition layers, and connectedness of the support of the limit measure (section \ref{section 3 phase field}). 

In section \ref{section computer} we show numerical evidence of the effectiveness and efficiency of our approach. To that end, we compare it to diffuse Willmore flow \mine{in two dimensions} without a topological term and to the penalisation proposed in \cite{Dondl:2011eh}. We have been unable to implement the functional developed in \cite{Dondl:2014vn} in practice, so a comparison with that could not be drawn. In section \ref{section conclusion} we discuss a few easy extensions of our main results.

\subsection{Topology and Phase Fields}

An often cited advantage of phase fields is that they are capable of changing their topology; in that sense our endeavour is non-standard. It should be noted that our phase fields may still change their topology (at least in three dimensions), only connectedness is enforced.

\mine{
Examples of topological changes and loss of connectedness in simulations for biological problems governed by bending energies or our type are given in \cite{Du:2007tt, Du:2010tt}.
}

\mine{
In \cite{Bonnivard:2014tw}, a geodesic distance function has been used to minimise the length of a connected set $K$ containing a prescribed set of points $x_1, \dots, x_N$ in two dimensions (Steiner's problem). Our setting is different in two ways: 1. Steiner's problem has a finite number of a priori known points which need to be contained in $K$ while the transition layer of the phase field has no special points and 2. the phase field approximation of Steiner's problem works in dimension $n=2$, while we work in ambient space of dimension $n=2,3$ where the curves used in the definition of the distance function have codimension $2$.
}

Previous work in~\cite{Dondl:2011eh} provides a first attempt at an implementation of a topological constraint in a phase field model for elastic strings modelled by the one-dimensional version of the Willmore energy, Euler's elastica. \mine{This technique prevents transitions in simulations for simple situations, but may fail in more complex cases, see section \ref{section computer}. Similar numerical approaches to controlling the topology using a diffuse Euler number are discussed in \cite{du:2005wm, du:2007ux}. }

This approach was complemented by a method put forth in \cite{Dondl:2014vn}, which relies on a second phase field subject to an auxiliary minimisation problem used to identify connected components of the transition layer. \mine{While our functional can be seen as using a diffuse measure of path-connectedness, the functional in \cite{Dondl:2014vn} generalises more directly the notion of connectedness.} For this model, a $\Gamma$-convergence result was obtained, showing that limits of bounded-energy sequences must describe a connected structure. Unfortunately, the complicated nested minimisation problem makes it unsuitable for computation.

Approaches of regularising limit interfaces have been developed by Bellettini in \cite{MR1459883} and investigated analytically and numerically in \cite{Esedoglu:2014vd}. The approaches work by introducing non-linear terms of the phase field in order to control the Willmore energies of the level sets individually and exclude transversal crossings (which phase fields for De Giorgi's functional can develop). These regularisations may prevent loss of connectedness along a gradient flow in practice, but do not lead to a variational statement via $\Gamma$-convergence.

\mine{
Furthermore, we would like to emphasise that we can easily describe a weakly* continuous evolution of varifolds along which connectedness is lost. Except at one singular time, the varifolds are embedded $C^2$-manifolds and the evolution is $C^2$-smooth. Details can be found in a forthcoming paper \cite{DW_sharp}.
}

\mine{
It thus is not clear whether the approach of \cite{MR1459883} does prevent topological transitions, in particular, the loss of connectedness, in three ambient space dimensions. At least, it is more difficult to implement due to the highly non-linear term including the Willmore energies of level sets.}

\mine{
Also in \cite{DW_sharp}, we show that topological genus is not continuous under varifold convergence and that minimising sequences of a constrained minimisation problem with fixed genus may change topological type in the limit.
}

\mine{
At this point, we thus know of no other model which can control the topology of phase field limits. Furthermore, our results are optimal since they allow us to control as much of the topology as can be controlled even for a sharp interface and they allow for efficient implementation.
}

\subsection{Biological Membranes}

\mine{
Willmore's and Helfrich's energies are widely used in the modelling of thin elastic structures. The first mention in that context goes back to Sophie Germain \cite{germain1821recherches}. Later, their importance has been suggested heuristically \cite{Helfrich:1973td} based on the principle that when a membrane is written as a graph over its tangent space, only derivatives of at most second degree should occur in at most quadratic expressions at the base point. Another biological motivation in the context of red blood cells is given in \cite{Canham:1970jp}.
}

\mine{
A Helfrich type functional has also been obtained as a macroscopic limit of certain mesoscale models for lipid bilayers \cite{peletier2009partial, lussardi2014variational}. Here bilayers are modelled using functions $u, v\in BV(\R^n, \{0,1\})$ to express locations of hydrophilic heads and hydrophobic tails of lipid molecules. These functions are coupled through a Monge-Kantorovich distance to be close together, constrained to satisfy $uv\equiv 0$ and the perimeter of $\{u=1\}$ is penalised. This energy prefers a bilayer structure and suitably rescaled versions $\Gamma$-converge to a Helfrich functional with $H_0 = 0$, $\chi_H = 1/2$ and $\chi_K = 1/3$.
}

\mine{
A heuristic way to motivate the occurrence of Willmore's energy in this context comes from thin shell theory. In \cite{friesecke2002theorem, friesecke2002rigorous} Friesecke, James and M\"uller proved $\Gamma$-convergence of non-linear three-dimensional elasticity to geometric bending energies including those of Willmore- or Helfrich-type in the vanishing thickness limit of thin plates. The admissible class here are isometric embeddings of domains in $\R^2$. 
The restriction to isometries stems from the fact that in-plane stretching energy scales with thickness of the plate and dominates out of-plane bending, which scales with the third power of the thickness parameter. Thus in the (second order) vanishing thickness limit, we are led to minimise Willmore's energy in a class of isometric immersions. The case of shells (i.e.\ non-flat structures) has been treated in \cite{friesecke2003derivation}.
}

\mine{
Lipid bilayers differ from shells in that they are liquid not solid and as such do not have a reference configuration. Assuming inextensibility, we must therefore also minimise over the space of Riemannian metrics on the bilayer with fixed area and the isometry constraint turns into the fixed area constraint.
}

\subsection{Further Context}

\mine{
Without any claim of completeness, let us give a bit more context of our topic. Willmore's energy is named for T.J.\ Willmore who studied it in a series of publications \cite{willmore1965note, willmore1971mean, willmore1992survey, willmore2000surfaces} and popularised it in his textbook \cite{willmore1996riemannian}. Independently of Willmore's work, the energy had already been considered as a bending energy for thins plates in \cite{germain1821recherches} and as a conformal invariant of surfaces embedded in $\R^3$ in \cite{MR0015247, thomsen}.
}

\mine{
As mentioned above, Willmore's energy is famously conformally invariant, in particular scale invariant. This shows that both the embedding constraint into $\Omega$ and the area constraint $\H^2 = S$ are needed to give a non-trivial constraint.
}

Extrema (in particular, local minimisers) of the Willmore functional are of interest in models for biological membranes, but they also arise naturally in pure differential geometry as the stereographic projections of compact minimal surfaces in $S^3$\mine{, see e.g.} \cite{pinkall1987willmore}, \mine{also for} examples. 

Even for Willmore's energy rigorous results are hard to obtain. From the point of view of the calculus of variations, a natural approach to energies as the ones above is via varifolds~\cite{Allard:1972vh}, \mine{\cite{MR825628}, where existence of minimisers for certain curvature functionals is proved}. The existence of smooth minimising Tori was proved by Simon \cite{Simon:1993uu}, and later generalised to surfaces of arbitrary genus in \cite{Bauer:2003er}. 

The long-standing Willmore conjecture that $\W(T) \geq 4\pi^2$ for all Tori embedded in $\R^3$ was recently established in \cite{marquesmin}, and the large limit genus of the minimal Willmore energy for closed orientable surfaces in $\R^3$ has been investigated in \cite{kuwert2010large}. The existence of smooth minimising surfaces under isoperimetric constraints has been established in \cite{schygulla2012willmore}. A good account of the Willmore functional in this context can be found in \cite{kuwert2012willmore}. 

The case of surfaces constrained to the unit ball was studied in \cite{Muller:2013vz} and a scaling law for the Willmore energy was found in the regimes of surface area just exceeding $4\pi$ and the large area limit. While the above papers adopt an external approach in the language of varifold geometry, a parametrised approach has been developed in \cite{riviere2014variational} and related papers. \mine{In \cite{MR3176354} this framework is used to solve the Willmore minimisation problem with prescribed genus and prescribed isoperimetric type. \cite{deckelnick2015minimising} gives a study of the Willmore functional on $C^2$-graphs and its $L^1$-lower semi-continuous envelope.}

Other avenues of research consider Willmore surfaces in more general ambient spaces \mine{\cite{MR2785762, MR2725514, MR2989995}}.

In the class of closed surfaces, \mine{if $\chi_K$ is constant,} the second term in the Helfrich functional is of topological nature due to the Gauss-Bonnet theorem. \text{So, i}f the minimisation problem is considered only among surfaces of prescribed topological type it can be neglected. The spontaneous curvature is realistically expected to be non-zero and can have tremendous influence. It should be noted that the full Helfrich energy depends also on the orientation of a surfaces for $H_0\neq 0$ and not only on its induced (unoriented) varifold. Gro\ss e-Brauckmann \cite{grosse1993new} gives an example of \mine{(obviously non-compact)} surfaces $M_k$ of constant mean curvature $H\equiv 1$ converging to a doubly covered plane. This demonstrates that, unlike the Willmore energy, the Helfrich energy need not be lower semi-continuous under varifold convergence for \mine{certain} parameters.

Recently, existence of minimisers for certain Helfrich-type energies among axially symmetric surfaces under an isoperimetric constraint was proved by Choksi \mine{and Veneroni}~\cite{Choksi:2013jw}. \mine{Lower semi-continuity for the Helfrich functional on $C^2$-boundaries with respect to the $L^1$-topology of the enclosed sets was established by means of Gauss graphs in \cite{delladio1997special}.}

Results for the gradient flow of the Willmore functional are still few. Short time existence for sufficiently smooth initial data \mine{has been shown in \cite{simonett2001willmore, kuwert2002gradient} (see also \cite{mayer2003self})} and long time existence for small initial energy \mine{and convergence to a round sphere has} been demonstrated in \cite{Kuwert:2001uh, kuwert2012willmore}\mine{. Kuwert and Sch\"atzle's lower bound on existence times in terms of initial curvature concentration in space has been generalised to Willmore flow in Riemannian manifolds of bounded geometry in \cite{link2013gradient}. I}t has been shown that Willmore flow can drive smooth initial surfaces to self-intersections in finite time in \cite{mayer2003self}. \mine{This issue seems to be prevented by our connectedness functional on the phase field level, although we do not have a rigorous statement on this.} Numerical simulations suggest that singularities can occur in  Willmore flow in finite time \cite{mayer2002numerical}. \mine{\cite[Figure 2]{Du:2007tt} gives a numerical example of a disc pinching off to a torus.} A level set approach to Willmore flow is discussed in \cite{droske2004level}. 

Studies of numerical implementations of Willmore flow are, for example, due to Garcke, Dziuk, Elliott \textit{et al.}~in~\cite{Barrett:2008gd,Dziuk:2008ib,Dziuk:2007gt}. Particularly interesting here is also an implementation of a two-step time-discretisation algorithm due to Rumpf and Balzani~\cite{Balzani:2012iq}.
 
Phase field approximations of the functional in~\eqref{eq:Will_Hel}, on the other hand, often provide a more convenient approach to gradient flows or minimisation of the Willmore or Helfrich functionals. 
In particular if a coupling of the surface to a bulk term is desired, phase field models can provide an excellent alternative to a parametric discretisation~\cite{Du:2010tt}.

The idea for applying \mine{a phase field} approach to Willmore's energy goes back to De Giorgi~\cite{DeGiorgi:1991jc}. For a slight modification of De Giorgi's functional, reading
\[
\E_\eps^\mathrm{dG}(u) = \frac{1}{c_0}\int_\Omega \frac{1}{\eps}\left( \eps\, \Delta u - \frac{1}{\eps}\,W'(u) \right)^2 \dx \:+ \mine{\frac{\lambda}{c_0}\, \int_\Omega \frac\eps2\,|\nabla u|^2 + \frac1\eps\,W(u)\dx} \:\:=: \mathcal{W}_\eps(u) + \lambda\, S_\eps(u)
\]
$\Gamma$-convergence to the sum of Willmore's energy and the $\lambda$-fold perimeter functional \stephan{($\lambda>0$)} was finally proved by R\"oger and Sch\"atzle~\cite{Roger:2006ta} in $n=2,3$ dimensions, providing the lower energy bound \stephan{in the varifold context}, after Bellettini and Paolini had provided the recovery sequence over a decade earlier ~\cite{Bellettini:1993vg}. First analytic evidence had previously been presented in the form of asymptotic expansions in \cite{Du:2005fx}. This and other phase field approximations of Willmore's energy and their $L^2$-gradient flows are reviewed in \cite{bretin2013phase}. A convergence result for a diffuse approximation of certain more general Helfrich-type functionals has also been derived by Belletini and Mugnai~\cite{Bellettini:2009ui}.

Regarding implementation of phase field models for Willmore's and Helfrich's energy, we refer to the work by Misbah \textit{et al.}~in~\cite{Biben:2005ds} and Du \textit{et al.}~\cite{Du:2005fx,du:2005wm,Du:2006hl,Du:2007tt,du:2007ux,Du:2009vl,Du:2010tt,Wang:2008ux}. In~\cite{Franken:2013tt}, the two-step algorithm for surface evolution has been extended to phase fields. A numerical implementation of the Helfrich functional can be found, for example, in the work by Campelo and Hernandez-Machado~\cite{Campelo:2006gf}.

\section{Phase Fields}\label{section review}

\subsection{General Background}\label{section 2 background}
A phase field approach is a method which lifts problems of $(n-1)$-dimensional manifolds which are the boundaries of sets to problems of scalar fields on $n$-dimensional space. Namely, instead of looking at $\partial E$, we study smooth approximations of $\chi_E-\chi_{E^c}$. \mine{Such models are generally based on a competition between a multi-well functional penalising deviation of the phase field function from the minima (usually $\pm 1$ instead of the characteristic function values $0$, $1$) and a gradient-penalising term preventing an overly sharp transition.} Fix $\Omega\cc\R^n$, open. Our model space is
\[
X = \{u\in W^{2,2}_{loc}(\R^n)\:|\:u\equiv -1 \text{ outside }\Omega\} \mine{= -1 + W^{2,2}_0(\Omega)}.
\]
If $\partial\Omega\in C^2$, this can be identified with  $X_a = \{u\in W^{2,2}(\Omega)\:|\: u = \partial_\nu u = 0 \text{ on }\partial\Omega\}$, otherwise the first formulation is technically easier to work with. The boundary conditions are chosen to model sets which are totally contained in $\Omega$ and whose boundaries may only touch $\partial\Omega$ tangentially. Then (as well as in more general settings) it is known \cite{Modica:1987us} that the \mine{Modica-Mortola or Cahn-Hilliard} functionals 
\[
S_\eps:L^1(\Omega)\to \R,\quad S_\eps(u) = \begin{cases}\frac1{c_0}\int_\Omega \frac\eps2\,|\nabla u|^2 + \frac1\eps\,W(u)\d x &u\in X\\ +\infty &\text{else,}\end{cases}
\]
with $W(u) = (u^2-1)^2/4$ and $c_0 = \int_{-1}^1 \sqrt{2\,W(u)\,}\d u = 2\sqrt{2\,}/3$, $\Gamma$-converge to
\[
S_0(u) = \frac12 |Du|(\overline\Omega)
\] 
with respect to strong $L^1$-convergence for functions $u\in BV(\R^n,\{-1,1\})$ (and $+\infty$ else). The limit $S_0(u)$ agrees with the perimeter functional $\mathrm{Per}(\{u=1\})$. The double-well potential $W$ forces sequences of bounded energy $S_\eps(u_\eps)$ to converge to a function of the above form as $\eps\to 0$\mine{. By Young's inequality $\frac{\eps}{4}a^2+\eps b^2\geq  ab$ we can control the $BV$-norms of $G(u_\eps)= \int_{-1}^{u_\eps} \sqrt{2\,W(s)\,}\d s$ \stephan{through
\begin{align*}
||\,G(u_\eps)||_1 &\leq C\,\left(\int_\Omega W(u_\eps)+1\right)^\frac34,\\
 \int_\Omega |\nabla G(u_\eps)|\dx& = \int_\Omega |G'(u_\eps)|\,|\nabla u_\eps|\dx\\
 	&\leq \int_\Omega \frac1{4\eps} \,|G'(u_\eps)|^2 + \eps\,|\nabla u_\eps|^2\dx\\
	&= 2\,S_\eps(u_\eps). 
\end{align*}
So $G(u_\eps)\to v$ converges strongly in $L^p(\Omega)$ for all $p<n/(n-1)$ for some $v\in BV(\Omega)$ by the compactness theorem for $BV$-functions (up to a subsequence), so in particular pointwise almost everywhere (further subsequence). $||u_\eps||_{L^4(\Omega)}^4 \leq C\,\big(1+ \frac1\eps\int_\Omega W(u_\eps)\dx\big)$, so $u_\eps\to u$ weakly in $L^4(\Omega)$ and $u_\eps \to G^{-1}(v)$ pointwise a.e.\ which together implies that  $u_\eps \to u = G\inv(v)$ strongly in $L^p(\Omega)$ for all $p<4$. T}he functional has originally been studied in the context of minimal surfaces, see e.g.\ \cite{Modica:1987us, Caffarelli:1995kh, Caffarelli:2006cp}. We will view them in a slightly different light. If $u_\eps$ is a finite energy sequence, we denote the Radon measure given by the approximations of the area functional as
\[
\mu_\eps(U) = \frac{1}{c_0}\int_U \frac\eps2\,|\nabla u_\eps|^2 + \frac1\eps\,W(u_\eps)\d x.
\]
Taking the first variation of this measure and integrating by parts, we get
\[
\delta\mu_\eps(\phi) = \frac1{c_0} \int_\Omega\left(-\eps\,\Delta u_\eps + \frac1\eps\,W'(u_\eps)\right)\,\phi\d x.
\]
\stephan{The integrand}
\[
v _\eps := -\eps\,\Delta u_\eps + \frac1\eps\,W'(u_\eps)
\]
\stephan{is} a diffuse analogue of the mean curvature as the variation of surface area. A modified version of a conjecture by De Giorgi then proposes that squaring $v_\eps$, integrating over $\Omega$ and dividing by $\eps$ should be a $\Gamma$-convergent approximation of $\W(\partial E)$ for $C^2$-regular sets $E$, again with respect to strong $L^1$-convergence. Dividing by $\eps$ is needed to convert the volume integral into a diffuse surface integral over the transition layer with width of order $\eps$. \stephan{This is true at least in $n = 2,3$ dimensions \cite{Roger:2006ta}.} From now on, we will restrict ourselves to these dimensions. We thus introduce the diffuse Willmore functional
\[
\W_\eps(u) = \begin{cases}\frac{1}{c_0}\int_\Omega\frac1\eps \left(\eps\,\Delta u - \frac1\eps\,W'(u)\right)^2\d x&u\in X\\ +\infty &\text{else}\end{cases}
\]
and the Radon measures associated with a sequence $u_\eps$ with bounded energy $S_\eps(u_\eps) + \W_\eps(u_\eps)$
\[
\mu_\eps(U) = \frac1{c_0}\int_U\frac\eps2\,|\nabla u_\eps|^2 + \frac1\eps\,W(u_\eps)\d x,\quad \alpha_\eps(U) = \frac1{c_0}\,\int_U\frac1\eps\, v_\eps^2\d x.
\]
Since the measures $\mu_\eps, \alpha_\eps$ are uniformly bounded, we can invoke
}
 the compactness theorem for Radon measures to see that there are finite Radon measures $\mu,\alpha$ with support in $\overline{\Omega}$ such that, up to a subsequence,
\[
u_\eps\to u\text{ strongly in }L^1(\Omega),\quad \mu_\eps\stackrel*\wto \mu,\quad \alpha_\eps\stackrel *\wto \alpha.
\]
\stephan{A finer comparison shows that $|D u|\leq 2\mu$.} Clearly by construction $u = \chi_E - \chi_{E^c}$ for some $E\subset \Omega$, and from the above argument is not difficult to see that $E$ is a Caccioppoli set. In three dimensions, however, $E$ may well be empty\stephan{. To see this, one can use the construction of \cite{Muller:2013vz}, where two concentric spheres of very similar radius are connected by a modified catenoid. As the catenoid is a minimal surface, this keeps the Willmore energy bounded. Taking the radii to converge to the same quantity and approximating them with the recovery sequence of a phase field, we see that $E$ is the $L^1$-limit of the enclosed set between the two original sphere, so empty.} 

In \cite{Roger:2006ta} it has been shown that $\mu$ is in fact the mass measure of an integral $n-1$-varifold with $|\partial^*E|\leq \mu$ and $ H^2_\mu\cdot \mu \leq\alpha$ for the generalised mean curvature $H_\mu$ of $\mu$. Unlike $\mu$, $\alpha$ can still behave wildly.

We aim to minimise the Willmore energy among connected surfaces with given area $S$. To prescribe area $\mu(\overline\Omega) = S$, we include a penalisation term of  
\[
\frac1{\eps^\sigma}\,(S_\eps - S)^2
\]
in the energy functional. \mine{So we have a good approximation of Willmore's energy available, the area constraint can be approximated through an easy penalisation, and confinement to $\Omega$ comes automatically from the choice of function space $X$. The main challenge left is controlling the connectedness of $\spt(\mu)$ which we interpret as a sharp interface membrane.
}

\subsection{Equipartition of Energy}\label{section 2 discrepancy}

\stephan{Intuitively, t}he Cahn-Hilliard energy \stephan{prefers} phase fields to be close to $\pm 1$ in phase and make transitions between the phases along layers of width proportional to $\eps$. \stephan{In particular, to attain the minimum, equality $|\nabla u_\eps| = |G'(u_\eps)|$ is needed in Young's inequality on most of $\Omega$. So it makes sense that the terms $\frac\eps2\,|\nabla u_\eps|^2$ and $\frac1\eps\,W(u_\eps)$ should make an equal contribution to the measure $\mu_\eps$.} The difference of their contributions is controlled by the discrepancy measures
\[
\xi_\eps(U) = \frac1{c_0}\int_U \frac\eps2\,|\nabla u_\eps|^2 - \,\frac1\eps\,W(u_\eps)\d x,
\]
their positive parts $\xi_{\eps,+}(U) = \frac1{c_0}\int_U \left(\frac\eps2\,|\nabla u_\eps|^2 - \frac1\eps\,W(u_\eps)\right)_+\d x$ and their total variation measures $|\xi_\eps|(U) = \frac1{c_0}\int_U \left|\frac\eps2\,|\nabla u_\eps|^2 - \frac1\eps\,W(u_\eps)\right|\d x$. We will see that for a suitable recovery sequence, $|\xi_\eps|$ vanishes exponentially in $\eps$; more generally the discrepancy goes to zero in general sequences $u_\eps$ along which $\alpha_\eps+\mu_\eps$ stays bounded as shown in \cite[Propositions 4.4, 4.9]{Roger:2006ta}. The result has been improved along subsequences in \cite[Theorem 4.6]{Bellettini:2009ui} to $L^p$-convergence for $p<3/2$ and convergence of Radon measures for the gradients of the discrepancy densities \mine{for an approximation of Helfrich's energy}.

The control of these measures is extremely important since they appear in the derivative of the diffuse $(n-1)$-densities $r^{1-n}\mu_\eps(B_r(x))$ \stephan{ with respect to the radius $r$}.

\subsection{Connectedness}\label{section 2 connectedness}

In order to ensure that the support of the limiting measure $\mu$ is connected we include an auxiliary term $C_\eps$ in the energy functional. \mine{The heuristic idea behind this is that if the support of the limiting measure $\spt(\mu)$ is connected, then so should the set $\{\rho_1<u_\eps<\rho_2\}$ for $-1<\rho_1<\rho_2<1$. These level sets away from $\pm 1$ can be heuristically viewed as approximations of $\spt(\mu)$, and in other situations, they can be seen to Hausdorff-converge to it \cite{Caffarelli:1995kh}. This is not quite true in our situation (in three dimensions, see Remark \ref{remark no uniform}), but the intuition still holds.}

\mine{Our concept is to introduce a quantitative notion of path-connectedness and penalise the measured disconnectedness. Take a weight function $F\in C^{0}(\R)$ such that
\[
F\geq 0,\qquad F\equiv 0\text{ on }[\rho_1,\rho_2], \qquad F(-1),\: F(1)>0 .
\]
and the associated geodesic distance
\[
d^{F(u)}(x,y) = \inf\left\{\int_K F(u)\d \H^1\:\bigg|\:K\text{ connected, }x,y\in K, \H^1(K)\le \omega(\eps)\right\}, \quad \textrm{$\omega(\eps) \to \infty$ as $\eps \to 0$}.
\]
\mine{In particular, $\omega(\eps)\equiv\infty$ is not excluded.} If $\{\rho_1<u_\eps<\rho_2\}$ is connected, we can connect any two points $x, y\in u_\eps\inv(\rho_1,\rho_2)$ by a curve of length zero. If it is not, then $d^{F(u)}(x,y)$ gives a quantitative notion of how badly path-connectedness fails between these two points. To obtain a global notion, we take a second weight function $\phi \in C_c(-1,1)$ resembling a bump, i.e.,
\[
\phi\geq 0,\qquad \{\phi>0\} = (\rho_1,\rho_2)\cc (-1,1),\qquad \int_{-1}^1\phi(u)\d u >0
\]
and take the double integral
\[
C_\eps(u) = \frac1{\eps^2}\int_\Omega\int_\Omega \phi(u(x))\,\phi(u(y)) \,d ^{F(u)}(x,y)\d x\d y.
\]
So, i}f $\{\phi(u_\eps)>0\} = \{\rho_1<u_\eps<\rho_2\}$ is connected, we can connect any two points $x,y\in \Omega$ such that $\phi(u_\eps(x))\,\phi(u_\eps(y))\,>0$ with a curve of length zero, hence $d^{F(u_\eps)}(x,y)=0$ and both the integrand and the double integral vanish.

If on the other hand $\spt(\mu)$ is disconnected, then we expect that $d^{F(u_\eps)}$ should be able to \stephan{discern} different connected components such that $\liminf_{\eps\to 0}C_\eps(u_\eps)>0$. The core part of our proof is concerned with precisely that. We need to show that $\phi$ detects components of the interface and that $d^{F(u_\eps)}$ \stephan{distinguishes} them. For the first result, we need to understand the structure of the interfaces converging to $\mu$ \stephan{and make sure they cannot be so steep in $u_\eps\inv(\rho_1,\rho_2)$ that the double integral does not see them in the limit.}

\mine{
For the second part, the challenge is to understand how phase fields converge away from the interface. It is clear that they converge pointwise almost everywhere, but it is not a priori clear whether they can stay away from $\pm 1$ along a curve in three dimensions. Namely, if the support of a limiting measure $\mu$ of $\mu_\eps$ had two components $C_1, C_2$, but the approximating sequence satisfied $u_\eps \in (\rho_1,\rho_2)$ in a neighbourhood of a curve $\gamma_\eps$ connecting $C_1$ and $C_2$, then $d^{F(u_\eps)}$ could not distinguish between the components.}

\stephan{
In two ambient dimensions, this is excluded relatively easily. In three-dimensional space however, the curve $\gamma_\eps$ would not be seen by the limiting measure $\mu$, which is an integral $2$-varifold. Thus we need to show that in analogy to the sharp interface, a long and narrow tunnel-like strucure around $\gamma_\eps$ blows up in diffuse Willmore energy. We like to think of this as a convergence localised on curves, i.e.\ sets of co-dimension $n-1 = 2$.
}

\subsection{Main Results}\label{section 2 main}

We have the following results for phase fields with bound\-ed Willmore energy and perimeter. In the following we assume that $u_\eps$ denotes the continuous representative.

\begin{theorem}\label{theorem n=2}
Let $n=2$ and $u_\eps\in X$ a sequence such that 
\[
\limsup_{\eps\to 0}\:(\W_\eps + S_\eps)\: (u_\eps) <\infty.
\]
Denote $\mu = \lim_{\eps\to 0}\mu_\eps$ for a subsequence along which \stephan{$\mu_\eps$ converge weakly as Radon measures to a limit} and take any $\Omega'\cc \R^n\setminus\spt(\mu)$. Then $|u_\eps|\to 1$ uniformly on $\Omega'$. 
\end{theorem}

\ton{This result has already been proved by Nagase and Tonegawa in \cite{nagase2007singular}, where they also show a $\Gamma-\liminf$-inequality for $\W_\eps + S_\eps$ in two space dimensions. The use of a monotonicity formula in their proof resembles ours with differences in the estimate on the discrepancy measures. We show our proof here to illustrate some ideas in the more involved argument for $n=3$.}

\begin{theorem}\label{Hausdorff n=2}
Let $n=2$ and $D\cc (-1,1)$. Then $u_\eps^{-1}(D)\to \spt(\mu)$ in the Hausdorff distance. If $n=3$, up to a subsequence, $u_\eps^{-1}(D)$ Hausdorff-converges to a set $K$ which contains $\spt(\mu)$.
\end{theorem}

\mine{
In three dimensions, $K\neq \spt(\mu)$ is possible, see Remark \ref{remark no uniform}. As an illustration, the following statement describes the convergence of phase fields on lower dimensional objects that we need. In the proofs a slightly different version of the theorem is used to account for changing curves as $\eps\to 0$.
\begin{theorem}\label{theorem n=3}
Let $n=2$ and $u_\eps\in X$ a sequence such that 
\[
\limsup_{\eps\to 0}\:(\W_\eps + S_\eps)\: (u_\eps) <\infty.
\]
Denote $\mu = \lim_{\eps\to 0}\mu_\eps$ for a subsequence along which {$\mu_\eps$ converges weakly as Radon measures to a limit}. Take any $K\cc \R^n\setminus\spt(\mu)$ such that $0<\H^1(K)<\infty$. Then $|u_\eps|\to 1$ converges $\H^1|_K$-almost everywhere. 
\end{theorem}
}
For our application to connectedness, we define the total energy of an $\eps$-phase field as 

\begin{equation}\label{our energy}
\E_\eps(u) = \begin{cases}\W_\eps(u) + \eps^{-\sigma} \, (S_\eps(u) -S)^2 + \eps^{-\kappa}\,C_\eps(u) &u\in X\\ +\infty &\text{else}\end{cases}\end{equation}
for $\sigma\in(0,4), \kappa>\mine{0}$. 
\begin{remark}
\mine{If $\omega(\eps)<\infty$, e}xistence of mininimisers for the functional $\E_\eps$ is a simple exercise in the direct method of the calculus of variations, since uniform convergence of a minimising sequence \mine{$u_{\eps, k}\in W^{2,2}(\Omega) \to C^0(\overline\Omega)$} for fixed $\eps$ guarantees convergence of the distance term.
\end{remark}

\mine{
\begin{theorem}\label{theorem ceps}
Let $u_\eps\in X$ be a sequence such that $(\W_\eps + S_\eps)(u_\eps) \leq C$ for some $C>0$ and $\mu$, $\alpha$ Radon measures such that $\mu_\eps \stackrel *\wto \mu$, $\alpha_\eps\stackrel *\wto \alpha$. If $\spt(\mu)$ is disconnected, then
\[
\liminf_{\eps\to 0} C_\eps(u_\eps)>0.
\]
\end{theorem}
}

Using these results, we can show the following.

\begin{theorem}
\label{main theorem}
 Let $n=2, 3$ and $u_\eps\in X$ a sequence such that $\liminf_{\eps\to 0}\E_\eps(u_\eps)<\infty$. Then the diffuse mass measures $\mu_\eps$ converge weakly* to a measure $\mu$ with connected support $\spt(\mu)\subset\overline\Omega$ and area $\mu(\overline\Omega) = S$.
\end{theorem}

Using \cite{Roger:2006ta}, $\mu$ is also the mass measure of an integral varifold. The main result of \cite{Roger:2006ta} can be applied to deduce $\Gamma$-convergence of our functionals in the following sense: 

\begin{corollary}\label{theorem gamma convergence}
Let $n=2,3$, $S>0$, $\Omega\cc\R^n$ and $E\cc\Omega$, with smooth boundary $\partial E\in C^2$ with area $\H^{n-1}(\partial E) = S$. Then
\[
\Gamma(L^1(\Omega))-\lim_{\eps\to 0} \E_\eps(\chi_{E}-\chi_{E^c}) = 
\left\{
\begin{array}{cl}
 \W(\partial E) \:  \label{GammaPres} & \partial E \text{ is connected}\\
 +\infty & \text{otherwise}
 \end{array}\right.
\]
\end{corollary}

Here we have adopted the notation of \cite{Roger:2006ta} where $\Gamma$-convergence is said to hold at a point if the $\liminf$- and $\limsup$-inequalities hold at that point. This distinction is necessary since it is not clear what $\Gamma$-convergence properties hold at other points $u\in BV(\Omega,\{-1,1\})$ since $\W_\eps$ does not converge to the $L^1$-lower semi-continuous envelope of $\W$ at non-embedded sets even in two dimensions. 

This issue stems from the fact that the stationary Allen-Cahn equation $-\Delta u + W'(u) = 0$ \stephan{admits global saddle-solutions which have vanishing sets asymptotic to a union of lines. The best known example vanishes along the coordinate axes and is positive in the first and third quadrants and negative in the second and fourth ones, see \cite{del2010multiple}.} These solutions can be used to approximate transversal crossings with zero Willmore energy, while the lower semi-continuous envelope of Willmore's energy/Euler's elastica energy becomes infinite at those points. Interestingly enough, the crossing has zero Willmore-energy as a varifold, so that the energy can still be justified, despite the fact that it does not give the sensible result in our situation. 

While heuristic considerations suggest -- and numerical simulations appear to confirm -- that at least in $n=2$ dimensions and when $0\notin \spt(\phi)$, our additional term in the energy might prevent saddles, we do not investigate convergence at non-embedded points further.

\begin{remark}\label{remark no uniform}
Uniform convergence generally fails in three dimensions. This can be seen by taking a sequence $u^\eps$ as the optimal interface approximation which will be given in the proof of Corollary \ref{theorem gamma convergence} and adding to it a perturbation $g((x-x_0)/\eps)$ for any $g\in C_c^\infty(\R^n)$. It can easily be verified that the energy will remain finite (at least if $|g|\ll 1$) but clearly the functions do not converge uniformly. 

Heuristic arguments show that for modifications $u_\eps = u^\eps + \eps^\alpha\, g(\eps^{-\beta}\cdot)$, the coefficients $\alpha = \beta = 1$ should be optimal in three dimensions and $\alpha = 1/2$, $\beta = 1$ if $n=2$, so we expect a convergence rate of $\sqrt{\eps}$ in the two-dimensional case.

\mine{Since uniform convergence does not hold, it follows that the pre-images of certain level sets can have additional points in the Hausdorff limit.}
\end{remark}

\begin{remark}
Since $S_\eps(u_\eps)$ is uniformly bounded, \stephan{a standard argument for the Modica-Mortola model shows} that $u_\eps\to \chi_E- \chi_{\Omega\setminus E}$ in $L^p$ for $p<4/3$. We are going to show in Lemma \ref{regularity lemma} that the sequence $u_\eps$ is uniformly bounded in $L^\infty(\R^n)$ when additionally $\limsup_{\eps\to 0} \W(u_\eps)<\infty$, which implies $L^p$-convergence for all $p<\infty$.
\end{remark}

\section{Proofs}\label{section proofs}

The proofs are organised in the following way. First, anticipating the results of section \ref{section 3 phase field}, we show $\Gamma$-convergence of $\E_\eps$ to $\W$. We decided to move the proof to the beginning since it is virtually independent of all the other proofs in this article and can be isolated, but introduces the optimal interface transitions which will be needed later. After that, we proceed chronologically with technical Lemmata (sections \ref{section 3 auxiliary 1} and \ref{section 3 auxiliary 2}) and the proof of Theorems \ref{theorem n=2}, \mine{\ref{theorem n=3},} \ref{Hausdorff n=2}, \mine{\ref{theorem ceps}} and \ref{main theorem} in section \ref{section 3 phase field}.

\subsection{\texorpdfstring{Proof of $\Gamma$-Convergence}{Proof of Gamma-Convergence}}\label{section 3 Gamma}
We now proceed to prove Corollary \ref{theorem gamma convergence}.

\begin{proof}[Proof of the $\liminf$-inequality:]
It will follow from Theorem \ref{main theorem} that $\E_\eps(u_\eps)\to \infty$ if $\partial E$ is disconnected. If $\partial E$ is connected, the main part of this inequality is to show that if $u_\eps\to \chi_E - \chi_{E^c}$ in $L^1(\Omega)$ and $\mu_\eps(\Omega)\leq S+1$, then $\liminf_{\eps\to 0}\E_\eps(u_\eps)\geq \W(\partial E)$. Since $\E_\eps\geq \W_\eps$ and enforces the surface area estimate, we obtain with \cite{Roger:2006ta} that
\[
\liminf_{\eps\to 0} \E_\eps(u_\eps)\geq \W(\partial E).
\]
\end{proof}

\begin{proof}[Proof of the $\limsup$-inequality:]
We may restrict our analysis to the case of connected boundaries with area $\H^{n-1}(\partial E) = S$. In this proof, we will construct a sequence $u^\eps\in X$ such that
\[
\lim_{\eps\to 0} ||u^\eps - (\chi_E-\chi_{E^c})||_{1,\Omega} = 0,\qquad \lim_{\eps\to 0}\E_\eps(u^\eps) = \W(\partial E).
\]
The construction is standard and holds for arbitrary $n\in\N$. We take the solution of the one-dimensional problem and apply it to the signed distance function of $\partial E$, thus recreating the shape of $\partial E$ with an (approximate) optimal profile for the transition from $-1$ to $1$. For $\delta>0$ consider
\[
U_\delta:= \{x\in \R^n\:|\:\dist(x,\partial E) <\delta\}.
\]
Since $E\cc \Omega$, $U_\delta\subset\Omega$ for all sufficiently small $\delta$, and since $\partial E\in C^2$ is embedded, there is $\delta>0$ such that
\[
\psi:\partial E\times (-\delta,\delta)\to U_\delta,\quad \psi(x,t) = x + t\,\nu_x
\]
is a diffeomorphism. \stephan{We have the closest point projection $\pi:U_\delta\to \partial E$ and the signed distance function $d(x) = \sdist(x,\partial E)$ from $\partial E$ which is positive inside $E$ and negative outside $E$. Both are $C^2$-smooth on $U_\delta$ and they satisfy $\psi(\pi(x),\sdist(x)) = x$ and} (see e.g.\ \cite[Section 14.6]{Gilbarg:2001vb}) 
\[
\nabla d(x) = \nu_{\pi(x)},\qquad \Delta d(x) = H_{\pi(x)} + C_{\pi(x)}\cdot d(x) + O(d(x)^2)
\] 
\stephan{where $C_{\pi(x)}$ is a constant depending on the base projection point $\pi(x)$.}
Let us consider the optimal transition between $-1$ and $1$ in one dimension. This optimal profile is a stationary point of $S_1$, i.e.\ a solution of $- q'' + W'(q) = 0$ (and thus a zero energy point of $\W_1$) with the side conditions that $\lim_{t\to \pm \infty}q(t) = \pm 1$. {Note that the optimal profile satisfies
\[
q'' = W'(q) \quad\Ra\quad q''\,q' = W'(q)\,q' \quad \Ra\quad \frac{d}{dt}\left((q')^2 - W(q)\right) = 0,
\]
so we find that $(q')^2 - W(q)\equiv c$. We look for transitions where both $(q')^2$ and $W(q)$ are integrable over the whole real line, and since $\lim_{t\to\pm\infty}W(q(t)) = 0$, we see that $c=0$ and $(q')^2 \equiv W(q)$. This gives us equipartition of energy already before integration.} For simplicity, we focus on $W(q) = (q^2-1)^2/4$ which has the optimal interface $q(t) = \tanh(t/\sqrt{2})$. Note that the functional rescales appropriately under dilations of the parameter space so that we have equipartition of energy before integration also in the $\eps$-problem. The disadvantage of the hyperbolic tangent is that it makes the transition between the roots of $W$ only in infinite space, so we choose to work with approximations $q_\eps\in C^\infty(\R)$ such that

\begin{enumerate}
\item $q_\eps(t) = q(t)$ for $|t|\leq \delta/(3\eps)$,
\item $q_\eps(t) = 1$ for $t\geq \delta/(2\eps)$,
\item $q_\eps(-t) = - q_\eps(t)$,
\item $q_\eps'>0$\stephan{,
\item $|q_\eps''|(t)\leq \eps^3$ for all $|t|\geq \delta/\eps$}.
\end{enumerate}
\stephan{Such a function is easily found as $\tanh$ converges to $\pm 1$ exponentially fast.} Then we set
\[
u^\eps(x) = q_\eps(d(x)/\eps).
\]
A direct calculation establishes that with this choice of $u^\eps$, we have $|S_\eps(u^\eps)-S|\leq \eps^\gamma$ for all $\gamma<2$, {$|\xi_\eps|\leq \eps^m$ for all $m\in\N$} and $\lim_{\eps\to 0}\W_\eps(u^\eps) = \W(\partial E)$. It remains to show that $\lim_{\eps\to 0}\eps^{-\kappa}C_\eps(u^\eps) = 0$. We will show that even $C_\eps(u^\eps)\equiv 0$ along this sequence. Since $\partial E$ is connected and $\psi$ is a diffeomorphism, all the level sets 
\[
\{u^\eps = \rho\} = \psi (\partial E, \eps\,q_\eps^{-1}(\rho))
\]
are connected manifolds for $\rho\in(-1,1)$ \mine{and $\eps$ small enough such that $q_\eps = q$ around $q\inv(\rho)$}. We know that
\[
\{\phi(u^\eps)>0\} = \{\rho_1<u^\eps<\rho_2\}
\]
and pick any $\rho\in(\rho_1,\rho_2)$. Now let $x,y\in\Omega$, $\phi(u^\eps(x)), \phi(u^\eps(y))>0$. We can construct a curve from $x$ to $y$ by setting piecewise\stephan{
\begin{align*}
\gamma_1:[0, d(x)]]\to\Omega\:&,\quad \gamma_1(t) = \pi(x) + t\,\nu_{\pi(x)},\\
\gamma_3:[0, d(y)]\to \Omega \:&,\quad \gamma_3(t) = \pi(y) + t\,\nu_{\pi(y)}
\end{align*}
and $\gamma_2$ any curve connecting $\pi(x)$ to $\pi(y)$ in $\{u^\eps=\rho\}$. This curve exists since connected manifolds are path-connected. The curve $\gamma = \gamma_3\oplus\gamma_2\oplus\gamma_1^{-1}$ connects $x$ and $y$} and satisfies by construction $\phi(\gamma(t))>0$, so $F(\gamma(t)) \equiv 0$. Therefore we deduce 
\[
d^{F(u^\eps)}(x,y)=0
\]
if $\phi(u^\eps(x)), \phi(u^\eps(y))\neq 0$ \mine{since} the connecting curves have uniformly bounded length and $\omega(\eps) \to \infty$. Thus in particular
\[
\frac1{\eps^2}\int_{\Omega\times\Omega}\phi(u^\eps(x))\,\phi(u^\eps(y))\, d ^{F(u^\eps)}(x,y)\d x\d y \equiv 0.
\]
\end{proof}

\subsection{Auxiliary Results I}\label{section 3 auxiliary 1}

A lot of our proofs will be inspired by \cite{Roger:2006ta} which again draws from \cite{Hutchinson:2000df}. We will generally cite \cite{Roger:2006ta} because it treats the more relevant case for our study. In this section, we will prove a number of auxiliary results which concern either general properties of phase fields or properties away from the support of the limiting measure $\mu$ which will enable us to investigate their convergence later. Results concerning phase interfaces are postponed until section \ref{section 3 auxiliary 2}. 

\mine{
We start with an optimal regularity lemma for phase fields on small balls. This is an improvement upon \cite[Proposition 3.6]{Roger:2006ta} where in three dimensions, only boundedness in $L^p$ for finite $p$ and a slow growth in $L^\infty$ could be obtained. The new $L^\infty$-bound also implies local H\"older continuity which is used extensively later.
}

\begin{lemma}\label{regularity lemma}
Let $n=2,3$, $\Omega\cc\R^n$ and $u_\eps\in X$ such that $\bar\alpha:= \limsup_{\eps\to 0}\W_\eps(u_\eps) < \infty$. Then there exists $\eps_0>0$ such that for all $\eps<\eps_0$

\begin{enumerate}
\item $||u_\eps||_{\infty,\R^n}\leq C$ where $C$ is a constant depending only on $\bar\alpha$ and $n$.

\item $u_\eps$ is $1/2$-H\"older continuous on $\eps$-balls, i.e.
\[
|u_\eps(x) - u_\eps(y)| \leq \frac{C}{\sqrt{\eps\,}}\,|x-y|^\frac12\:\qquad\forall\ x\in\R^n,\:\:y\in B(x,\eps).
\]
Again, the constant $C$ depends only on $\overline \alpha$ and $n$.
\end{enumerate}
\end{lemma}

\begin{proof}
We will argue using Sobolev embeddings for blow ups of $u_\eps$ onto the natural length scale. In the first step, we show the set $\{|u_\eps|>1\}$ to be small. In the second step, we estimate the $L^2$-norm of the blow ups, in the third we estimate the full $W^{2,2}$-norm and use suitable embedding theorems to conclude the proof of regularity. 

{\bf Step 1.} First observe that
\begin{align} \nonumber
\alpha_\eps(\Omega) &\geq \alpha_\eps(\{u_\eps> 1\})\\ \nonumber
	&= \frac1{c_0}\int_{\{u_\eps>1\}} \frac1\eps\,\left(\eps\,\Delta u_\eps - \frac1\eps\,W'(u_\eps)\,\right)^2\d x\\ \nonumber
	&= \frac{-2}{c_0\,\eps}\int_{\partial \{u_\eps > 1\}}W'(u_\eps)\,\partial_\nu u_\eps\d \H^{n-1} \\ \nonumber
&\qquad +\frac1{c_0}\int_{\{u_\eps>1\}}\eps\,(\Delta u_\eps)^2 + \frac2\eps\,W''(u_\eps)\,|\nabla u_\eps|^2 + \frac1{\eps^3}\,(W'(u_\eps)\,)^2\d x\\ \label{eq W'3}
	&\geq \frac1{c_0}\int_{\{u_\eps>1\}}\eps\,(\Delta u_\eps)^2 + \frac4\eps\,|\nabla u_\eps|^2 + \frac1{\eps^3}\,(W'(u_\eps)\,)^2\d x
\end{align}
using that $W''(t)\geq 2$ for $t\geq 1$. The boundary integral vanishes if $\{u_\eps>1\}$ is of finite perimeter since $u_\eps\in W^{2,2}\embeds C^{0,1/2}$ is continuous and $W'(1) = 0$. If this is not the case, take $\theta \searrow 1$ converging from above such that $\{u_\eps>\theta\}$ is of finite perimeter. This holds for almost all $\theta\in \R$. The sign of the boundary integral can be determined since $\partial_\nu u_\eps < 0$ on the boundary of $\{u_\eps>\theta\}$ and $W'(\theta)>0$ for $\theta>1$ so that the same inequality can still be established. By symmetry, the same argument works for $\{u_\eps<-1\}$.

{\bf Step 2.} Let $x_\eps\in \Omega$ be an arbitrary sequence and define the blow up sequence $\tilde u_\eps:\R^n\to \R$ by
\[
\tilde u_\eps(y) = u_\eps(x_\eps + \eps y).
\]
Then we observe that
\begin{align*}
\int_{B(0,2)}\tilde u_\eps^2 \d x &= \int_{B(0,2)} \left(\,|\tilde u_\eps|-1+1\right)^2\d y\\
	&\leq \int_{B(0,2)} \left(\,(|\tilde u_\eps|-1)_++1\right)^2\d y\\
	&\leq 2\,\int_{B(0,2)}(\,|\tilde u_\eps|-1)_+^2 + 1\d y\\
	&\leq 2\,\eps^{3-n} \int_{\{|u_\eps|>1\}} \frac1{\eps^3}\,W'(u_\eps)^2\d y + 2^{n+1}\,\omega_n \\
	&\leq 2\,\left\{2^n\,\omega_n + c_0\,\eps^{3-n}\,\alpha_\eps(\Omega)\,\right\}.
\end{align*}
As usual, $\omega_n$ denotes the volume of the $n$-dimensional unit ball. In exactly the same way with a slightly simpler argument we obtain
\[
\int_{B(0,2)} (W'(\tilde u_\eps))^2 \d y \leq C(\bar\alpha,n).
\]
{\bf Step 3.} Now a direct calculation shows that 
\begin{align*}
\int_{B(0,2)}\left(\Delta\tilde u_\eps - W'(\tilde u_\eps)\right)^2\d y
	&= \int_{B(0,2)} (\eps^2\Delta u_\eps - W'(u_\eps))^2\,(x_\eps+\eps y)\d y\\
	&= c_0 \eps^{3-n}\,\alpha_\eps(B(x_\eps,2\eps)).
\end{align*}
Thus
\begin{align*}
||\,\Delta \tilde u_\eps\,||_{2,B(0,2)} &\leq ||\,\Delta \tilde u_\eps - W'(\tilde u_\eps)\,||_{2,B(0,2)} + ||\,W'(\tilde u_\eps)\,||_{2,B(0,2)}\\
	&\leq \sqrt{c_0\,\eps^{3-n}\,\alpha_\eps(\Omega)} + \sqrt{C(\bar\alpha,n)}.
\end{align*}
In total, we see that
\[
||\,\tilde u_\eps\,||_{2, B(0,2)} + ||\,\Delta\tilde u_\eps\,||_{2,B(0,2)} \leq C(\bar\alpha, n) 
\]
for all $0<\eps<1$ so small that $\alpha_\eps(\Omega)\leq \alpha(\Omega)+1$. Using the elliptic estimate from \cite[Theorem 9.11]{Gilbarg:2001vb}, we see that 
\[
||\tilde u_\eps||_{2,2, B(0,1)} \leq C(\bar\alpha, n),
\]
where we absorb the constant depending only on $n$ and the radii into the big constant. Using the Sobolev embeddings
\[
W^{2,2}(B(0,1)) \embeds W^{1,6}(B(0,1)) \embeds C^{0,1/2}(\overline{B(0,1)})
\]
we deduce that 
\[
|\tilde u_\eps|_{0,1/2, B(0,1)}\:\leq\:C(n,\bar \alpha),
\]
again absorbing the embedding constants into the constant. In particular, this shows that 
\[
||\tilde u_\eps||_{\infty, B(0,1)}\leq C(n,\bar \alpha).
\]
But since this holds for all sequences $x_\eps$, we can deduce that
\[
||\,u_\eps\,||_{\infty, \R^n}\,\leq\, C(n,\bar\alpha).
\]
Furthermore, for $x, z\in \R^n$ with $|x-z|<\eps<\eps_0$, we choose $x_\eps = x$ to deduce 
\[
|u_\eps(x) - u_\eps(y)| = |\tilde u_\eps(0) - \tilde u_\eps(\,(y-x)/\eps)| \leq C(n,\bar\alpha)\,\left|\,(y-x)/\eps\,\right|^\frac12 = \frac{C(n,\bar\alpha)}{\sqrt{\eps\,}}\,|x-y|^\frac12.
\]
\end{proof}

\begin{remark}
Without prescribing boundary conditions as in our modified space $X$, the result could still be salvaged on compactly contained subsets. Techniques for estimating quantities over $\{|u_\eps|>1\}$ in that case can be found in \cite[Proposition 3.5]{Roger:2006ta}, which we include below for the readers' convenience.
\end{remark}
\begin{prop}\cite[Proposition 3.5]{Roger:2006ta}\label{RS_prop_bdry}
For $n=2,3$, $\Omega \subseteq \R^n$, $\eps>0$, $u_\eps\in C^2(\Omega)$, $v_\eps \in C^0(\Omega)$, 
\[
-\eps \Delta u_\eps + \frac{1}{\eps} W'(u_\eps) = v_\eps \quad \text{in $\Omega$},
\]
and $\Omega'\cc \Omega$, $0 < r < \operatorname{dist}(\Omega', \partial\Omega)$, we have
\[
\int_{\{|u_\eps| \ge 1\}\cap \Omega'} W'(u_\eps)^2 \le C_k(1 + r^{-2k}\eps^{2k})\eps^2\int_\Omega v_\eps^2+ C_k r^{-2k}\eps^{2k} \int_{\{|u_\eps| \ge 1\}\cap \Omega} W'(u_\eps)^2
\] 
for all $k\in \N_0$.
\end{prop}

\mine{Thus $u_\eps$ is uniformly finite on compactly contained sets for small $\eps$ and steps 2 and 3 from the proof of Lemma \ref{regularity lemma} go through as before.} A useful rescaling property is the following observation from the proof of \cite[Theorem 5.1]{Roger:2006ta}.

\begin{lemma}\label{rescaling lemma}
Let $u_\eps:B(x,r)\to \R$, $\lambda>0$ and $\hat u_\eps: B(0,r/\lambda)\to \R$ with
\[
\hat u_\eps(y) = u_\eps(x+\lambda y).
\]
Set $\hat r:= r/\lambda$, $\hat \eps:= \eps/\lambda$, 
\[
\hat\mu_\eps:= \frac1{c_0}\: \left(\frac{\hat \eps}2\,|\nabla \hat u_\eps|^2 + \frac1{\hat\eps}\,W(\hat u_\eps)\right)\L^n,\qquad \hat \alpha_\eps:= \frac1{c_0\,\hat\eps}\,\left(\hat \eps\,\Delta \hat u_\eps - \frac1{\hat\eps}\,W'(\hat u_\eps)\right)\,\L^n.
\]
Then 
\[
\hat r^{1-n}\hat\mu_\eps(B(0,\hat r)) = r^{1-n}\,\mu_\eps(B(x,r)),\qquad \hat r^{3-n}\, \hat\alpha_\eps(B(0,\hat r)) = r^{3-n}\,\alpha_\eps(B(x,r)).
\]
The discrepancy measures $\xi_{\eps,\pm}$ behave like $\mu$ under rescaling. 
\end{lemma}

\begin{proof} 
This can be seen by a simple calculation similar to the one in the proof of Lemma \ref{regularity lemma}.
\end{proof}

For the reader's convenience we include the following classical monotonicity result.

\begin{lemma}\cite[Lemma 4.2]{Roger:2006ta}\label{lemma monotonicity formula}
For $x\in\R^n$ we have
\begin{align*}
\frac{d}{d\rho}\left(\rho^{1-n}\,\mu_\eps(B(x,\rho))\right) &= - \frac{\xi_\eps(B(x,\rho))}{\rho^n} + \frac1{c_0\,\rho^{n+1}}\int_{\partial B(x,\rho)} \eps\,\langle y-x,\nabla u_\eps\rangle^2\d \H^{n-1}(y)\\
	&\qquad\qquad + \frac1{c_0\,\rho^n}\int_{B(x,\rho)}v_\eps\,\langle y-x, \nabla u_\eps\rangle \d y.
\end{align*}
\end{lemma}

In low dimensions $n=2,3$, the second and third term in the monotonicity formula can easily be estimated after integration. While the result is known, we fixed \mine{some} details in the proof of \cite[Proposition 4.5]{Roger:2006ta}, so we include it here for completeness.

\begin{lemma} \cite[Proposition 4.5]{Roger:2006ta}\label{lemma estimated monotonicity formula}
Let $0<r<R<\infty$ if $n=3$ and $0<r<R\leq 1$ if $n=2$, then
\begin{align}\label{eq monotonicity formula}\nonumber
r^{1-n}\mu_\eps(B(x,r)) &\leq 3\,R^{1-n}\mu_\eps(B(x,R)) + 2\,\int_{r}^R\frac{\xi_{\eps,+}(B(x,\rho))}{\rho^n}\d \rho  \\ 
	&\nonumber \quad + \frac{1}{2\,(n-1)^2}\alpha_\eps(B(x,R)) + \frac{r^{3-n}}{(n-1)^2}\,\alpha_\eps(B(x,r)) \\
	&\quad + \frac{R_0^2\,R^{1-n}}{(n-1)^2}\,\alpha_\eps(B(x,R))
\end{align}
where $R_0:= \min\{R,R_\Omega\}$ and $R_\Omega$ is a radius such that $\Omega\subset B(0,R_\Omega/2)$
\end{lemma}

\begin{proof}
Without loss of generality we may assume that $x=0$ and write $B_\rho:= B(0,\rho)$, $f(\rho)= \,\rho^{1-n}\mu_\eps(B_\rho)$. Observe that for any function $g:B_R\to\R$ we have 
\begin{align*}
\int_r^R\rho^{-n}\,\int_{B_\rho} g(x)\d x\d \rho 
	&= \int_{B_R}g(x)\,\int_{\max\{|x|, r\}}^R \rho^{-n}\d \rho\d x\\
	&= \frac1{n-1}\int_{B_R}g(x) \,\left(\frac1{\max\{|x|,r\}^{n-1}} - \frac1{R^{n-1}}\right)\d x
\end{align*}
and
\begin{align*}
\int_r^R\rho^{-(n+1)}\int_{\partial B_\rho} g(x)\d \H^{n-1}\d \rho
	&= \int_{B_R\setminus B_r} \frac{g(x)}{|x|^{n+1}}\d x.
\end{align*}
Using this to integrate the derivative we obtain using Young's inequality
\begin{align*}
f(R)&-f(r) = \int_r^Rf'(\rho)\d \rho\\
	&= \int_r^R \frac{-\,\xi_\eps(B_\rho)}{\rho^{n}}\d \rho + \frac1{c_0}\int_{B_R\setminus B_r} \frac{\eps\,\langle \nabla u_\eps,y\rangle^2}{|y|^{n+1}} + \frac{1}{n-1}\, \frac{v_\eps\,\langle y,\nabla u_\eps\rangle}{|y|^{n-1}}\d y\\
		&\qquad +  \frac1{(n-1)\,c_0\,r^{n-1}} \int_{B_r}v_\eps\,\langle y,\nabla u_\eps\rangle \d y -  \frac1{(n-1)\,c_0\,R^{n-1}} \int_{B_R}v_\eps\,\langle y,\nabla u_\eps\rangle  \d y\\
	&\geq \int_r^R\frac{-\,\xi_{\eps,+}(B_\rho)}{\rho^n}\d \rho \\
		&\qquad + \frac1{c_0}\int_{B_R\setminus B_r}\frac{\eps\,\langle \nabla u_\eps,y\rangle^2}{|y|^{n+1}} - \frac1{n-1}\left((n-1)\,\eps\,\frac{\langle y,\nabla u_\eps\rangle^2}{|y|^{2(n-1)}} + \frac{1}{4\,(n-1)\,\eps}\,v_\eps^2\right)\d y\\
		&\qquad - \frac{1}{c_0\,r^{n-1}}\int_{B_r}\lambda\,\frac{\eps\,\langle y,\nabla u_\eps\rangle^2}{2\,|y|^2} + \frac1{2\lambda}\frac{\,|y|^2\,v_\eps^2}{(n-1)^2\,\eps}\d y \\
		&\qquad - \frac{1}{c_0\,R^{n-1}}\int_{B_R}\lambda\,\frac{\eps\,\langle y,\nabla u_\eps\rangle^2}{2\,|y|^2} + \frac1{2\lambda} \frac{\,|y|^2\,v_\eps^2}{(n-1)^2\,\eps}\d y\\
	&\geq \int_r^R\frac{-\,\xi_{\eps,+}(B_\rho)}{\rho^n}\d \rho - \frac1{4\,(n-1)^2}\int_{B_R\setminus B_r}\frac1\eps\,v_\eps^2\d y\\
	&\qquad- \lambda\,f(r) - \frac1{2\lambda}\,\frac{r^2}{(n-1)^2\:r^{n-1}}\,\int_{B_r} \frac{v_\eps^2}{\eps}\d y\\
&\qquad - \lambda\,f(R) - \frac1{2\lambda}\,\frac{R_0^2}{(n-1)^2\,R^{n-1}}\int_{B_R}\frac{v_\eps^2}{\eps}\d y.
\end{align*}
where $\lambda\in(0,1)$. Here we used that $n=2,3$ to obtain that $2(n-1)\leq n+1$, so that $|y|^{n+1} \leq |y|^{2(n-1)}$ for all $|y|$ if $n=3$ and for $|y|\leq 1$ if $n=2$. When we bring all the relevant terms to the other side, this shows that
\begin{align*}
(1+\lambda)\,f(R) - (1-\lambda)\,f(r) &\geq -\int_r^R\frac{ \xi_{\eps,+}(B_\rho)}{\rho^n}\d \rho - \frac{1}{4\,(n-1)^2}\,\alpha_\eps(B_R\setminus B_r)\\
	&\qquad - \frac{r^{3-n}}{2\,\lambda\,(n-1)^2}\,\alpha_\eps(B_r) - \frac{R_0^2}{2\lambda\,(n-1)^2\,R^{n-1}}\,\alpha_\eps(B_R).
\end{align*}
Setting $\lambda= 1/2$ and multiplying by two proves the Lemma.
 \end{proof}
 
\begin{remark}\label{remark li-yau}
If $n=3$, we may let $R\to\infty$ and subsequently $\eps\to 0$, $r\to 0$ and finally $\lambda\to 0$ that we have
\[
\limsup_{r\to 0} r^{1-n}\mu(B(x,r)) \leq \frac1{4\,(n-1)^2}\,\alpha(\overline\Omega)
\]
at every point $x\in\R^3$ such that $\alpha(\{x\}) = 0$ (i.e.\  when $\lim_{r\to 0}\alpha(B_r)=0$). Using the results of \cite{Roger:2006ta}, $\mu$ is an integral varifold, so this yields a Li-Yau-type \cite{li1982new} inequality
\begin{equation}\label{Li-Yau}
\theta^*(\mu,x) = \limsup_{r\to 0}\frac{\mu(B(x,r))}{\pi\,r^2} \leq \frac1{16\,\pi}\alpha(\overline\Omega).
\end{equation}
This inequality is usually found with a $4$ in place of the $16$ which stems from a different normalisation of the mean curvature and $\W(\mu)$ in the place of $\alpha$, see also the proof of \cite[Lemma 1]{MR1650335} or \cite[Proposition 2.1.1]{kuwert2012willmore}.

In $n=2$ dimensions, \stephan{we had to assume $R\leq 1$}. Indeed, an inequality of this type cannot hold since circles with large enough radii have arbitrarily small elastic energy. Still, setting $R=1$, a similar bound on the multiplicity in terms of $\bar\alpha$ and $S$ can be obtained.
\end{remark}

\stephan{The version of the monotonicity formula \eqref{eq monotonicity formula} which we will use} is the simplified expression
\begin{equation}\label{our monoticity}
r^{1-n}\mu_\eps(B(x,r))\leq 3\,R^{1-n}\mu_\eps(B(x,R)) + 3\,\alpha_\eps(B(x,R)) + 2\,\int_r^R\frac{\xi_{\eps,+}(B(x,\rho))}{\rho^n}\,d\rho.
\end{equation}
This holds generally if $n=3$, and when $R\leq 1$ if $n=2$. Furthermore, we have the following estimate for the positive part of the discrepancy measures. \mine{It is a precise quantitative refinement of the classic statement that smooth solutions of the stationary Allen-Cahn equation $-\Delta u + W'(u) = 0$ on $\R^n$ satisfy $|\nabla u|^2 \leq 2\,W(u)$ \cite{MR803255}.}

\begin{lemma}\cite[Lemma 3.1]{Roger:2006ta}\label{lemma RS estimate}
Let $n=2,3$. Then there are $\delta_0>0, M\in\N$ such that for all $0<\delta\leq \delta_0$, $0<\eps\leq \rho$ and 
\[
\rho_0:= \max\{2, 1+ \delta^{-M}\eps\}\,\rho
\]
we have
\begin{align*}
\rho^{1-n}\xi_{\eps,+}(B(x,\rho))&\leq C\,\delta\,\rho^{1-n}\,\mu_\eps(B(x,2\rho)) + C\,\delta^{-M}\eps^2\,\rho^{1-n}\int_{B(x,\rho_0)}\frac1\eps\,v_\eps^2\d x\\
	&\qquad+ C\,\delta^{-M}\eps^2\,\rho^{1-n}\int_{B(x,\rho_0)\cap\{|u_\eps|>1\}}\frac1{\eps^3}W'(u_\eps)^2\d x + \frac{C\,\eps\,\delta}\rho.
\end{align*}
\end{lemma}

The following observation is the key ingredient in order to obtain our required convergence results. \mine{Geometrically, it expresses that phase fields $u_\eps$ which stay away from $\pm 1$ at a fixed point need to have a minimal positive measure $\alpha$ or $\mu$. In three dimensions, this is only proved under additional technical assumptions. The proof uses infinitesimal H\"older continuity and the monotonicity formula to transition from the $\eps$-length scale to the macroscopic scale. 
}

In a slight abuse of notation, we will denote the functionals defined by the same formulas by $\W_\eps, S_\eps$ again, although they are given on spaces over $B_1:=B(0,1)$ \stephan{instead of $\Omega$.}

\begin{lemma}\label{minimisation lemma}
Let $n=2,3$, $\theta\in (0,1)$, $0<\eta<1/2$. Consider the subsets 
\[
Y^2:= \{u\in W^{2,2}(B_1)\::\:|u(0)|\leq \theta \}
\]
in $n=2$ dimensions and 
\[
Y^3_\eps:= \left\{u\in W^{2,2}(B_1):\:|u(0)|\leq \theta\text{ and }\alpha_\eps(B_{\sqrt\eps}) + \int_{B_{\sqrt\eps}\cap\{|u_\eps|>1\}}\frac{W'(u)^2}{\eps^3}\d x \leq \eps^\eta \right \}.
\]
Define $\F_\eps:W^{2,2}(B_1)\to [0,\infty)$ as
\[
\F_\eps(u) = \W_\eps(u) + S_\eps(u).
\]
Then $\theta_0:= \liminf_{\eps\to 0} \inf_{u\in Y^2} \F_\eps(u) > 0$ if $n=2$ and
\[
\theta_0:= \liminf_{\eps\to 0} \inf_{u\in Y_\eps^3} \F_\eps(u) > 0
\]
if $n=3$. The same works if instead $u(0)\geq 1/\theta$.
\end{lemma}
\begin{proof}[Proof:] By H\"older continuity, the condition that $|u(0)|\leq \theta$ leads to the creation of an infinitesimal diffuse mass density $\eps^{1-n}\mu_\eps(B_\eps)\geq c_{n,\bar\alpha, \theta}$. We will use the monotonicity formula to integrate this up to show that if $\alpha=0$, macroscopic mass is created as well. In three dimensions, there is an additional technical complication which forces us to make two steps, one from the $\eps$-scale to the length scale of $\sqrt{\eps}$ and a second one to the original scale.

{\bf Step 1.} In a first step we show that the diffuse mass densities are uniformly bounded on large enough length scales. For $x\in\overline\Omega$, set $f_\eps(\rho):= \rho^{1-n}\mu_\eps(B(x,\rho))$. Without loss of generality, we may assume that $f_\eps(1)\leq 1$ for small enough $\eps>0$ and that $\alpha_\eps(B_{3/4}) + \int_{B_{3/4}\cap \{|u_\eps|>1\}}\frac1{\eps^3}W'(u_\eps)^2\d x\leq 1$ due to \cite[Proposition 3.5]{Roger:2006ta} (included here as Proposition~\ref{RS_prop_bdry}).

Take $\delta = \log(\eps)^{-2}$ in Lemma \ref{lemma RS estimate} to obtain from Lemma \ref{lemma estimated monotonicity formula} that for $0<r<R=1$ we have
\begin{align*}
f_\eps(r) &\leq 3\,f_\eps(R) + 3\,\alpha_\eps(B_{R}) + 2\int_r^{R}\frac{\xi_{\eps,+}(B_\rho)}{\rho^n}\d\rho\\
	&\leq C_{\bar\alpha,n} + C\int_{r}^{R} \frac{1}{\log(\eps)^2} \frac{f_\eps(2\rho)}{\rho}+ \frac{\eps}{\log(\eps)^2\rho^2}\\
	&\hspace{3cm} + \eps^2\,\log(\eps)^{2M}\rho^{-n}\left(\alpha_\eps(B_\rho) + \int_{B_\rho\cap\{|u_\eps|>1\}} \frac{W'(u_\eps)^2}{\eps^3}\d x\right) \d\rho\\
	&\leq C_{\bar\alpha,n} + \frac{C}{\log(\eps)^2}\int_r^R \frac{f_\eps(2\rho)}\rho\d\rho + \frac{C\eps}{\log(\eps)^2}\left(\frac1r - \frac1R\right)\\
	&\qquad + \frac{C\,\eps^2\,\log(\eps)^{2M}}{n-1}\left(r^{1-n} - R^{1-n}\right)\\
	&\leq C_{\bar\alpha,n} + \frac{C}{\log(\eps)^2} \int_r^{2R}\frac{f_\eps(\rho)}\rho\d\rho
\end{align*}
for a uniform constant $C_{\bar\alpha,n}$ for $r\geq \eps$ if $n=2$ and for $r\geq \sqrt{\eps}$ if $n=3$. We use Gr\"onwall's inequality backwards in time to deduce that
\[
f_\eps(r) \leq C_{\bar\alpha,n}\,\exp\left( \frac{C}{\log(\eps)^2}\int_r^{2R}\frac1\rho\d\rho\right) \leq C_{\bar\alpha,n}
\]
on $[\eps,1)$ if $n=2$ and on $[\sqrt{\eps},1)$ if $n=3$.

{\bf Step 2.} If $n=3$ and additionally 
\[
\alpha_\eps(B_{\sqrt\eps}) + \int_{B_{\sqrt\eps}\cap\{|u_\eps|>1\}}\frac{W'(u)^2}{\eps^3}\d x \leq \eps^\eta,
\]
we can estimate the terms in the second line more sharply to obtain uniform boundedness of $f_\eps(r)$ also for $\eps\leq r\leq R=\sqrt{\eps}$.

{\bf Step 3.} Now we turn to the proof of the statement. For a contradiction, assume that $(\alpha_\eps + \mu_\eps)(B_1)\to 0$ for a suitable sequence $u_\eps$. The functions $u_\eps$ are $C^{0,1/2}$-H\"older continuous with H\"older constant $C/\sqrt{\eps}$ for a uniform $C\geq 0$ on $B_{1/2}$, as can be obtained like in Lemma \ref{regularity lemma}. The only difference is that we need to use \cite[Proposition 3.5]{Roger:2006ta} (included here as Proposition~\ref{RS_prop_bdry}) to estimate $\int_{\{|u_\eps|>1\}}\frac1{\eps^3}W'(u_\eps)^2\d x$ due to the lack of boundary values. It follows that $u(x)\leq \frac{1+\theta}2$ for $x\in B(0,c\eps)$ for some uniform $c>0$, which implies 
\[
\eps^{1-n}\mu_\eps(B_\eps)\geq \frac1{\eps^n}\int_{B(0,c\eps)}W\left(\frac{1+\theta}2\right)\,dx =: c_{n,\bar\alpha, \theta}>0.
\]
In the following, we will assume $n=3$. The two-dimensional case follows with an easier argument of the same type. First we deduce that
\begin{align*}
f_\eps(\eps) &\leq 3\,f_\eps(\sqrt{\eps}) + 3\,\alpha_\eps(B_{\sqrt\eps})\\
	&\qquad +  C \int_\eps^{\sqrt\eps}\frac1{\log(\eps)^2\rho} + \frac{\eps}{\log(\eps)^2\,\rho^2} + \eps^2\,\log(\eps)^{2M}\rho^{-n}\,\eps^\beta\d\rho.
\end{align*}
Since the second line goes to zero as $\eps\to 0$ and $\alpha_\eps\wto 0$ by assumption, we deduce that $f_\eps(\sqrt\eps)\geq c_{n,\bar\alpha,\theta}/4$ for all sufficiently small $\eps>0$. Finally, we obtain
\begin{align*}
f_\eps(\sqrt{\eps}) &\leq 3\,f_\eps(1) + 3\,\alpha_\eps(B_1)\\
	&\qquad + C\int_{\sqrt\eps}^1 \frac1{\log(\eps)^2\rho} + \frac{\eps}{\log(\eps)^2\,\rho^2} + \eps^2\,\log(\eps)^{2M}\rho^{-n}d\rho.
\end{align*}
Again, the terms in the second line vanish with $\eps\to 0$ and $\alpha_\eps(B_1)\to 0$ by assumption, thus
\[
f_\eps(1)\geq \frac{c_{n,\bar\alpha,\theta}}{16}
\]
for all sufficiently small $\eps>0$. But this contradicts the assumption that $(\alpha_\eps+\mu_\eps)(B_1)\to 0$, so we are done. The case $u(0)\geq1/\theta$ follows similarly in two dimensions\stephan{; in three dimensions it is automatically excluded by the choice of space $Y_\eps^3$.}
\end{proof}

The estimate above is not sharp, but suffices for our purposes. After this Lemma, everything is in place to show uniform convergence of phase fields in two dimensions in section \ref{section 3 phase field}, while further results will be needed for Hausdorff convergence of the transition layers. In a very weak phrasing, Lemma \ref{lemma no mass at the edge} suffices, more precise versions (and the application to connectedness) need the entire section \ref{section 3 auxiliary 2}.

\subsection{Auxiliary Results II}\label{section 3 auxiliary 2}

In this section, we will derive technical results concerning how phase field approximations interact with the function $\phi$ as needed for the functional $C_\eps$ to impose connectedness. While the previous section focused on estimates away from the interface, here we investigate the structure of transition layers close to $\spt(\mu)$. The following Lemma is a special case of \cite[Proposition 3.4]{Roger:2006ta} with a closer attention to constants and the limit $\eps\to 0$ already taken.

\begin{lemma}\label{lemma no mass at the edge}
Let $x\in\R^n$, $r, \delta>0$, $0< \tau < 1- 1/\sqrt{2}$. Then
\[
\limsup_{\eps\to 0} \mu_\eps\big(\{u_\eps\geq 1-\tau\}\cap B(x,r)\big) \leq 4\,\tau\,\mu(B(x,r+\delta)).
\]
\end{lemma}

For all $x\in\R^n$, there are only countably many radii $r>0$ such that $\mu(\partial B(x,r))>0$. This follows from the fact that $\mu$ is finite and that there are at most finitely many radii such that $\mu(\partial B(x,r))\geq 1/k$, so that the union of those sets is countable. Thus for any $r>0$ there is $t\in(0,r)$ such that $\mu(\partial B(x,t)) = 0$. Letting $\delta\to 0$ at such a radius $t$ (and using that the discrepancy measures go to zero) gives us the following result (compare also \cite[Lemma 9]{Dondl:2014vn}).

\begin{corollary}\label{some phi}
For all $x\in\spt(\mu)$, $r>0$ and $\tau < 1/8$ we have 
\[
\liminf_{\eps\to 0}\frac1\eps\,\L^n\left(\,\{|u_\eps|<1-\tau \}\cap B(x,r)\}\right)>0.
\]
\end{corollary}

The following arguments rely more on the rectifiable structure of the measure $\mu$ that we are approximating. Specifically, we introduce the diffuse normal direction by
\[
\nu_\eps:= \frac{\nabla u_\eps}{|\nabla u_\eps|}
\]
when $\nabla u_\eps\neq 0$ and $0$ else. To work with varifolds, we introduce the Grassmannian $G(n,n-1)$ of $n-1$-dimensional subspaces of $\R^n$. We refer readers unfamiliar with varifolds or countably rectifiable sets to the excellent source \cite{simon1983lectures}; an introduction with \mine{different} focus which is easier to find and covers most results relevant for us is \cite{krantz2008geometric}. Recall the following result.

\begin{lemma}\cite[Propositions 4.1, 5.1]{Roger:2006ta}\label{proposition varifold convergence}
Define the $n-1$-varifold $V_\eps:= \mu_\eps \otimes \nu_\eps$ by 
\[
V_\eps(f) = \int_{\R^n\times G(n,n-1)} f(x, \langle \nu_\eps\rangle ^\bot)\d \mu_\eps\qquad\forall\ f\in C_c(\R^n\times G(n,n-1)).
\]
Then there is an integral varifold $V$ such that $V_\eps\to V$ \stephan{weakly} as Radon measures on $\R^n\times G(n,n-1)$ (varifold convergence). The limit satisfies
\[
\mu_V = \mu,\qquad H_\mu^2\,\mu\leq \alpha
\]
where $\mu_V$ is the mass measure of $V$ and $H_\mu$ denotes the generalised mean curvature of $\mu$. In particular, $\W(\mu)\leq \bar \alpha$. 
\end{lemma}

The following result is a suitably adapted version of \cite[Proposition 5.5]{Roger:2006ta} for our purposes. It shows that given small discrepancy measures and small oscillation of the gradient, a bounded energy sequence looks very much like an optimal interface in small balls. Using our improved bounds from Lemma \ref{regularity lemma}, we can drop most of their technical assumptions.

\begin{lemma}\label{lemma local transition}
Let $\delta, \tau>0$ and denote $\nu_{\eps, n}= \langle \nu_\eps, e_n\rangle$. Then there exist $0<L <\infty$ depending on $\delta$ and $\tau$ only and $\gamma>0$ depending on $\bar\alpha, \delta$ and $\tau$ such that the following holds for all $x\in\R^n$. If 

\begin{enumerate}
\item $|u_\eps(x)|\leq 1-\tau$ and
\item $|\xi_\eps|(B(x, 4L\eps)) + \int_{B(x, 4L\eps)}1- \nu_{\eps,n}^2\d \mu_\eps \leq \gamma\,(4L\eps)^{n-1}$ 
\end{enumerate}
then also \stephan{the following two properties hold:}

\begin{itemize}
\item \stephan{ The blow up $\tilde u_\eps(y) = u_\eps(x+\eps y)$ is $C^{0,1/4}$-close to an optimal profile $q$ on $B(0, 3L)$: 
\[
|\pm \tilde u_\eps - q(y_n -  t_1)|_{0, 1/4, B(0,3L)} < \delta.
\]
The optimal profile $q$ is the function from the $\limsup$-construction and $t_1= q^{-1}(u_\eps(x))$. }
\item $|u_\eps(\hat x,x_n+ t)|\geq 1-\tau/2$ for all $L\eps\leq |t|\leq 3L\eps$, where $\hat x = (x_1,\dots, x_{n-1})$ and $u$ changes sign in between.
\end{itemize}
\end{lemma}
\begin{proof}
Without loss of generality, we may assume that $x=0$ and write $B_r:= B(0,r)$. Recall that $q'(t) = \sqrt{2\,W(q(t))\,}$ and $\lim_{t\to \pm \infty}q(t) = \pm1$. Thus we can pick $L>0$ such that $|q(t)|\geq 1-\tau/4$ for all $t>L$. 

Assume for a contradiction that there is no constant $\gamma>0$ such that the results of the Lemma hold. Then for $\gamma^j\to 0$, there must be a sequence $u_\eps^j$ such that $|u_\eps^j(0)|\leq 1-\tau$, $\W_\eps(u_\eps^j)\leq\bar \alpha+1$ and 
\[
|\xi^j_\eps|(B_{4L\eps}) + \int_{B_{4L\eps}}1- (\nu_{\eps,n}^j)^2\d \mu_\eps\leq \gamma^j\,(4L\eps)^{n-1},
\]
 but the conclusions of the Lemma do not hold. Considering the blow ups $\tilde u^j: B_{4L}\to \R$ with $\tilde u^j(y) = u_\eps^j(\eps y)$ we obtain
 \[
||\tilde u^j||_{2,2,B_{3L}}\leq C_{\bar\alpha, n, L}
\]
like in Lemma \ref{regularity lemma}\stephan{. H}ence there is $\tilde u\in W^{2,2}(B_{3L})$ such that
\[
\tilde u^j\wto \tilde u\quad\text{ in }W^{2,2}(B_{3L}).
\]
Since $W^{2,2}$ embeds compactly into $W^{1,2}$ and $L^4$, we see that
\begin{align*}
\int_{B_{3L}} \big|\,|\nabla \tilde u|^2/&2 - W(\tilde u)\,\big|\d x = \lim_{j\to \infty} \int_{B_{3L}} \left|\,|\nabla \tilde u^j|^2/2 - W(\tilde u^j)\,\right|\d x\\
	&\leq \lim_{j\to \infty}\eps^{1-n}|\xi_\eps^j|(B_{4L\eps})\\
	&\leq \liminf_{j\to \infty}\, (4L)^{n-1}\gamma^j\\
	&= 0
\intertext{and when we set $\hat\nabla u = (\partial_1 u, \dots,\partial_{n-1}u)$, we get}
\int_{B_{3L}}|\hat\nabla \tilde u|\d x &= \lim_{j\to \infty} \int_{B_{3L}}|\hat\nabla \tilde u^j|\d x\\
	&= \lim_{j\to \infty}\int_{B_{3L}} \sqrt{|\nabla \tilde u_j|^2 - |\partial_n\tilde u^j|^2\,}\d x\\
	&\leq \liminf_{j\to\infty} \int_{B_{4L}} |\nabla \tilde u^j|\,\sqrt{1 - \left(\tilde\nu^j_n\right)^2\,}\d x\\
	&\leq \liminf_{j\to\infty} \left(\omega_n\,(4L)^n\right)^{1/2} \left(\int_{B_{4L}} |\nabla \tilde u^j|^2\,\left(1-(\tilde \nu^j_n)^2\right)\d x\right)^\frac12\\
	&\leq \liminf_{j\to\infty} \sqrt{8L\,\omega_n} \left(\eps^{1-n}\int_{B_{4L\eps}} 1- (\nu^j_n)^2\d \mu_\eps\right)^\frac12\\
	&\leq \liminf_{j\to\infty} \sqrt{8L\,\omega_n\,\gamma^j\,}\\
	&=0.
\end{align*}
Thus we can see that
\[
|\nabla \tilde u|^2 = 2\,W(\tilde u),\qquad \nabla \tilde u = (0, \dots, 0, \partial_n\tilde u).
\]
Clearly, this means that $\tilde u (y) = p(y_n)$ for a function $p$ with $p' = \pm \sqrt{2\,W(p)\,}$. Using that $|\tilde u(0)|\leq 1-\tau$ and the Picard-Lindel\"off theorem on the uniqueness of the solutions to ODEs, we see that $p(y_n) = \pm q(y_n-\bar y)$ for some $\bar y\in\R$ which can easily be fixed by the initial condition for $p(0)$.

Since weak $W^{2,2}$-convergence implies strong $C^{0,1/4}$-convergence in $n=2,3$ dimensions, we see that there is $j\in\N$ such that the claim of the Lemma holds for $u_\eps^j$ contradicting our assumption. Thus the Lemma is proven.
\end{proof}

To deal with the rectifiable sets in the next section more easily we prove a structure result for rectifiable sets. The result seems standard, but we have been unable to find a reference for it. As usual, we call a function on a closed set differentiable if it admits a differentiable extension to a larger open set.

\begin{lemma}\label{lemma structure rectifiable}
Let $M$ be a countably $k$-rectifiable set in $\R^n$. Denote by $B$ the closed unit ball in $k$ dimensions. Then there exist injective $C^1$-functions $f_i:B\to \R^n$ with $\nabla f_i\neq 0$ on $B$ such that 
\[
\H^k\left(M\setminus \bigcup_{i=1}^\infty f_i(B)\right) =0
\]
and such that $f_i(B)\cap f_j(B)=\emptyset$ for all $i\neq j$.
\end{lemma}

\begin{proof}
According to \cite[Lemma 5.4.2]{krantz2008geometric} or \cite[Lemma 11.1]{simon1983lectures} there is a countable collection of $C^1$-maps $g_i:\R^k\to\R^n$ such that
\[
M \subset N\cup \bigcup_{i=1}^\infty g_i\left(\R^k\right)
\]
where $\H^k(N) = 0$. Without loss of generality, $N$ is assumed to be disjoint from the other sets. First we need to make the individual maps $g_i$ one-to-one. To do that, we define the set where injectivity fails in a bad way:
\[
A_i:= \left\{x\in\R^k\:|\:\forall\ r>0\:\exists\ y\in B(x,r)\text{ such that }g_i(x)=g_i(y)\:\right\}.
\]
Due to the failure of local injectivity, we see that the Jacobian $J_{g_i}(x)$ vanishes on $A_i$. Since $g_i$ is a $C^1$-function, the set $D_i := J_{g_i}^{-1}(0)$ is closed and by the Morse-Sard Lemma \cite[3.4.3]{F} 
\[
\H^k\left(g_i(D_i)\right)  =    0.
\]
Set $U_i:= \R^k\setminus D_i$. Now as in \cite[Chapter 1.5, Corollary 2]{Evans:1992uy} we can use Vitali's covering theorem \cite[Chapter 1.5, Theorem 1]{Evans:1992uy} to obtain a countable selection of closed balls $B_i^j$ such that $f_i$ is injective with non-vanishing gradient on $B_i^j$ for all $j\in\N$ and
\[
\L^k\left(U_i\setminus\bigcup_{j=1}^\infty B_i^j\right)=0.
\]
Since the boundary of a $k$-ball has Hausdorff dimension $k-1$, we could equally well take open balls. Since $C^1$-functions map sets of $\L^k$-measure zero to sets of $\H^k$-measure zero, we have shown that we can write 
\[
M \subset \tilde N \cup\bigcup_{j=1}^\infty \tilde g_i(B^\circ)
\]
where $\H^k(\tilde N) = 0$,  $\tilde g_i:B\to\R^n$ is one-to-one, $C^1$, and has a non-vanishing gradient everywhere on the closed ball $B$. The functions $\tilde g_m$ are obtained by rescaling suitable restrictions of $g_i$ from $B_i^j$ to the unit ball. Finally, we have to cut out the sets that get hit by more than one function $\tilde g_m$. Inductively, we define
\[
\tilde U_m := B^\circ \setminus \tilde g_m^{-1}\left(\bigcup_{l=1}^{m-1} \tilde g_l(B)\right).
\]
Finally, we use Vitali's Lemma again to pick collections of closed balls $\tilde B_m^l$ such that 
\[
\L^k\left(\tilde U_m \setminus \bigcup_{l=1}^\infty\tilde B_m^l\right) = 0.
\]
Rescaling the restricted functions from these balls and translating to the unit ball gives us the result.
\end{proof}

The proof of the following Lemma resembles that of the integrality of $\mu$ in \cite[Lemma 4.2]{Roger:2006ta}. \mine{It faces different challenges: while we do not need to prove multi-layeredness, we cannot zoom in on the tangent space since we need a macroscopic measure contribution to the double integral. Thus we need Lemma \ref{lemma structure rectifiable} to approximate macroscopically the structure of $\mu$.
}

\begin{lemma}\label{lemma phi sees mu}
Let $\phi\in C^0(\R)$ such that $\phi\geq 0$ and $\int_{-1}^1\phi(u)\d u >0$. If $x\in \spt(\mu)$, then
\[
\liminf_{\eps\to 0}\:\frac1\eps\,\int_{B(x,r)} \phi(u_\eps)\d x>0
\]
for all $r>0$.
\end{lemma}

\begin{proof}
\textbf{Step 1.} As usual, we assume that $x=0$, $\mu(\partial B(x,r/2)) = 0$ and denote $B=B(x,r/2)$. This means that all the $\eps$-balls of positive integral we are going to find will actually lie in $B(x,r)$ and is a purely technical condition. Let $\zeta$ be a small constant to be specified later. For further use, denote by $\hat B$ the closed unit ball in $\R^{n-1}$. 

As $\mu$ is an integral varifold, we know that $\spt(\mu)$ is rectifiable. \stephan{This means by Lemma \ref{lemma structure rectifiable} that there are countably many $C^1$-functions $f_i:\hat B\to\R^n$ such that }
\[
\spt(\mu) \subset M_0 \cup \bigcup_{i=1}^\infty f_i(\hat B),\quad \H^{n-1}(M_0) = 0\:\quad f_i(\hat B)\cap f_j(\hat B) = \emptyset
\]
for $i\neq j$. Since $\mu$ has second integrable mean curvature $H_\mu^2\cdot\mu\leq\alpha$, we can further use the  Li-Yau inequality \mine{\eqref{Li-Yau}} to bound the maximum multiplicity of $\mu$ uniformly by
\[
\theta_{\max}\leq \frac{\alpha(\overline\Omega)}{16\,\pi} ,
\]
at least $\H^{n-1}$-almost everywhere. Now since $\H^{n-1}(\spt(\mu))< +\infty$ we can find $N\in\N$ such that 
\[
\H^{n-1}\left(\big(\spt(\mu)\cap B\big) \setminus\bigcup_{i=1}^Nf_i(\hat B) \right) < \frac{\zeta}{\,\theta_{\max}}.
\]
Since $f_i$ is injective and has non-vanishing tangent maps everywhere, $M:= \bigcup_{i=1}^Nf_i\left(\hat B^\circ\right)$ is a $C^1$-manifold. We observe that
\[
\H^{n-1}(\spt(\mu)\cap B\setminus M) < \frac{\zeta}{\theta_{\max}}
\]
and hence
\[
\mu\left(B\setminus M\right) < \zeta.
\]
Since the maps in question are smooth and the unit discs are orientable, for every $i$ we can pick a continuous unit normal field to $f_i(\hat B)$ (e.g.\ using cross products). Since the discs are compact and disjoint (thus a positive distance apart), the fields defined on each disc separately induce a continuous unit vector field on the union of their closures. 

Now we use the Tietze-Urysohn extension theorem to obtain a vector field $X$ on $B$ such that $X= \nu_M$ on $M$ and projecting on the unit ball we ensure $|X|\leq 1$. After an easy modification, we may assume that $|X|=1$ on a neighbourhood of $M$. We then define 
\[
G: \R^n\times G(n,n-1)\to \R,\qquad G(x,S) = \langle X_x, \nu_S\rangle^2
\]
where $\nu_S$ is one of the unit normals to $S$. Note that $G$ is continuous since $X$ is. Using the non-negativity of $G$ and the fact that $T_x\mu = T_xM$ for $\H^{n-1}$-almost every $x\in M\cap\spt(\mu)$ we interpret $\mu$ as dual to $C^0(\R^n\times G(n,n-1))$ and observe 
\begin{align*}
\langle \mu, G\rangle &= \int_{\spt(\mu)} \theta(x)\, G(x,T_x\mu)\d \H^{n-1}\\
	&\geq \int_{\spt(\mu)\cap M} \theta(x)\,G(x, T_xM)\d \H^{n-1} \\
	&= \int_{\spt(\mu)\cap M}\theta(x)\d \H^{n-1}\\
	&= \mu(M)\\
	&\geq \mu(B)-\zeta.
\end{align*}
\textbf{Step 2.} By varifold convergence, we know that $\lim_{\eps\to 0} \langle \mu_\eps, G\rangle = \langle \mu, G\rangle \geq \mu(B) - \zeta$, and $|X|, |\nu_\eps|\leq 1$ so
\begin{align*}
\limsup_{\eps\to 0}\int_B\big|1-\langle \nu_\eps,X\rangle^2\big|\d \mu_\eps 
	&= \limsup_{\eps\to 0}\int_B 1-\langle \nu_\eps,X\rangle^2 \d \mu_\eps\\
	&\leq \limsup_{\eps\to 0}\left(\mu_\eps(B) - \langle \mu_\eps, G\rangle \right) \\
	&\leq \zeta.
\end{align*} 
For $\gamma,\eps,L>0$ we define the set
\[
U_{\eps,\gamma,L} := \left\{x\in B\:\bigg|\:\frac{1}{(4L\eps)^{n-1}}\int_{B(x,4L\eps)}\big|1- \langle \nu_\eps, X\rangle^2\,\big| \d \mu_\eps >\gamma/4\right\}.
\]
Let $x_1,\dots, x_K$ be points in $U_{\eps,\gamma,L}$ being maximal for the property that the balls $B(x_i, 4L\eps)$ are disjoint. Then by definition
\begin{align*}
\zeta \geq \int_B\big|1-\langle \nu_\eps,X\rangle^2\big|\d \mu_\eps &\geq \sum_{i=1}^K \int_{B(x_i, 4L\eps)}\big|1-\langle \nu_\eps,X\rangle^2\big|\d \mu_\eps \geq K\,(4L\eps)^{n-1}\,\gamma/4.
\end{align*}
At the same time, we know that the balls $B(x_i, 8L\eps)$ cover $U_{\eps,\gamma,L}$ because otherwise we could bring in more disjoint balls, therefore
\begin{align*}
\frac{\L^n(U_{\eps,\gamma,L})}\eps &\leq K\frac{\,\omega_n\,(8L\eps)^{n}}\eps\\
	&\leq \frac{4\,\zeta}{\,\gamma\,(4L\eps)^{n-1}}\,\frac{\omega_n\,(8L\eps)^n}{\eps}\\
	&= 2^{n+4}\,\omega_n\,L\,\zeta/\gamma.
\end{align*}
For a given $\gamma$, we choose $\zeta = \zeta(\gamma)$ such that this is $\leq \mu(B)/4$. 

{\bf Step 3.} Knowing that $|\xi_\eps|(B)\to 0$, we can use the same argument as in the second step to show for
\[
V_{\eps,\gamma,L}:= \left\{x\in B\:\bigg|\:\frac{|\xi_\eps|\,(B(x,4L\eps))}{(4L\eps)^{n-1}}>\gamma/2\right\}
\]
the estimate
\[
\frac{\L^n(V_{\eps,\gamma,L})}\eps \leq \mu(B)/4
\]
for all sufficiently small $\eps>0$. 

{\bf Step 4.} Now choose $U$ \stephan{as a neighbourhood} of $M$ on which $|X|=1$ and $\tau>0$ like in Corollary \ref{some phi} satisfying 
\[
\liminf_{\eps\to 0} \mu_\eps\left(U\cap \{|u_\eps|\leq 1-\tau\}\right) \geq \frac{3\,\mu(B)}4.
\]
This is easily achieved when $\mu(M) > 3\,\mu(B)/4$. Furthermore we take $\delta\ll 1$ suitably small for small deviations of the optimal interface to behave similarly enough, $L$ and $\gamma$ as in Lemma \ref{lemma local transition} and $\zeta = \zeta(\gamma)$. Using steps one through three, we see that 
\begin{align*}
\liminf_{\eps\to 0}&\frac{\L^n\left(\{|u_\eps|\leq 1-\tau\}\cap U \setminus\left( U_{\eps,\gamma,L} \cup V_{\eps,\gamma,L}\right)\right)}\eps\\
	&\geq \liminf_{\eps\to 0}\frac{\L^n\left(\{|u_\eps|\leq 1-\tau\}\cap U\right)}\eps - \frac{\L^n\left(U_{\eps,\gamma,L}\right)}\eps - \frac{\L^n\left(V_{\eps,\gamma,L}\right)}\eps\\
	&\geq 3\,\mu(B)\,/4 - \mu(B)/4 - \mu(B)/4\\
	&=\mu(B)/4.
\end{align*}
\stephan{Using the reverse argument of step 2}, we can see that there are at least $K$ points $x_1,\dots,x_K$ in $\{|u_\eps|\leq 1-\tau\}\cap U\setminus\left( U_{\eps,\gamma,L} \cup V_{\eps,\gamma,L}\right)$ such that the balls $B(x_i, 4L\eps)$ are disjoint with 
\[
K\geq \frac{\mu(B)}{8^{n+1}\,L^n\,\eps^{n-1}}.
\]
{\bf Step 5.} To apply Lemma \ref{lemma local transition}, we must ``freeze'' the coefficients of the vector field $X$ to a single unit vector. We compute 
\begin{align*}
\frac1{(4L\eps)^{n-1}}&\bigg|\int_{B(x_i,4L\eps)}\left(1- \langle \nu_\eps, X\rangle^2\right) - \left(1 - \langle \nu_\eps, X_i\rangle^2\right) \:d\mu_\eps\bigg|\\
	&= \frac1{(4L\eps)^{n-1}}\bigg|\int_{B(x_i,4L\eps)} \langle \nu_\eps, X_i\rangle^2 -  \langle \nu_\eps, X\rangle^2 \:d\mu_\eps\bigg|\\
	&= \frac1{(4L\eps)^{n-1}}\bigg|\int_{B(x_i,4L\eps)} \langle \nu_\eps, X_i - X\rangle\,\langle \nu_\eps, X_i+X\rangle \:d\mu_\eps\bigg|\\
	&\leq |X_i + X|_{C^0(B(x_i,4L\eps))} \cdot |X_i - X|_{ C^0(B(x_i,4L\eps))}\frac1{(4L\eps)^{n-1}}\cdot \int_{B(x_i,4L\eps)}\:d\mu_\eps\\
	&\leq 2\,C_{\bar\alpha, L, n} \, |X_i - X|_{ C^0(B(x_i,4L\eps))}
\end{align*}
for all $X_i$ such that $|X_i|\leq 1$. When we set $X_i = X(x_i)$, the last term converges to zero -- so eventually it is smaller than $\gamma/4$ and 
\[
\frac1{(4L\eps)^{n-1}} \int_{B(x_i, 4L\eps)} 1- \langle \nu_\eps, X_i\rangle^2\:d\mu_\eps < \gamma/2.
\]
Since $x_i\in U$, we finally see that $|X_i|=1$ and Lemma \ref{lemma local transition} can be applied. 

{\bf Step 6.} Since $u_\eps$ is $C^{0,1/4}$-close to a one-dimensional optimal profile on $B(x_i, 3L\eps)$ which transitions from $-1$ to $1$, we see that for each $s\in (-(1-\tau), (1-\tau))$ there must be a point $y_i \in B(x_i, 3L\eps)$ such that $u_\eps(y_i) = s$. By H\"older continuity, we deduce that
\[
\int_{B(x_i, 3L\eps)} \phi(u_\eps)\:dx \geq \bar \theta \,\eps^n
\]
for a constant $\bar \theta$ depending on the support of $\phi$ and on $\bar \alpha, n$ for the H\"older constant. Since the balls are disjoint by construction, we can add this up to 
\begin{align*}
\frac1\eps\int_{B}\phi(u_\eps)\:dx &\geq \frac1\eps\,\sum_{j=1}^M\int_{B(x_i, 3L\eps)} \phi(u_\eps)\:dx\\
	&\geq \frac1\eps\,M\,\bar\theta\,\eps^{n}\\
	&\geq \frac{\mu(B)\,\bar\theta}{\:8^{n+1}\,L^n\:}\\
	&>0.
\end{align*}
This concludes the proof.
\end{proof}

We need one final result from geometric measure theory before we move on to our main results. 

\mine{
\begin{lemma}\label{lemma measure support}
Let $(X,d)$ be a metric space and $K\subset X$ is compact. If $K$ is not connected, then there exist two open sets $U_1, U_2\subset X$ such that
\[
K\subset U_1\cup U_2, \qquad K\cap U_i\neq \emptyset\text{ for } i=1,2\qquad\text{and}\qquad \dist(U_1,U_2)>0.
\]
\end{lemma}
}

\mine{
\begin{proof}
Assume that $\spt(\mu)$ is not connected. Then there exist relatively open non-empty sets $W_1, W_2\subset K$ such that 
\[
K = W_1\cup W_2.
\]
By definition of the subspace topology, $W_1$ and $W_2$ are also relatively closed. Since $K$ is compact, they are even compact, so $\delta:= \dist(W_1, W_2)>0$. We set $U_1= \{\dist(\cdot, W_1)< \delta/3\}$ and $U_2= \{\dist (\cdot, W_2)<\delta/3\}$.
\end{proof}
}

\mine{
If $K$ is the support of a measure $\mu$, we can pick $x_i\in W_i$ for $i=1,2$ since neither set is empty. As $U_i$ is a neighbourhood of $x_i$, we see directly from the definition of the support of a measure that $\mu(U_i)>0$ for $i=1,2$.
}

\subsection{Proof of the Main Results}\label{section 3 phase field}

Having dealt with the necessary auxiliary results, we can proceed to prove our main results. We use the terminology of Lemma \ref{minimisation lemma}.

\begin{proof}[Proof of Theorem~\ref{theorem n=2}]
\ton{Compare also \cite[Lemmas 3.1, 3.2 and Proposition 4.2]{nagase2007singular}.}
Let $\Omega'\cc \R^n\setminus\spt(\mu)$ and $\tau>0$. Assume there is a sequence $x_\eps$ in $\Omega'$ such that $|u_\eps(x_\eps)|\leq 1-\tau$ or $|u_\eps(x_\eps)|\geq 1+\tau$. By definition, $x_\eps\in \Omega$ and using compactness, there is $x\in \overline{\Omega' \cap\Omega}$ such that $x_\eps\to x$.

Let $r>0$ such that $B(x,3r)\subset \R^n\setminus\spt(\mu)$. Due to convergence, $B(x_\eps,r)\subset B(x,2r)$ for all sufficiently small $\eps>0$. If we use the rescaling property from Lemma \ref{rescaling lemma} and the minimisation property of Lemma \ref{minimisation lemma} for $n=2$, we see that
\begin{align*}
\alpha(\overline{B(x,2r)}) &\geq \frac1r\,\liminf_{\eps\to 0} \left(r\,\alpha_\eps(B(x_\eps,r)) + r^{-1}\,\mu_\eps(B(x_\eps,r))\right) \\
	&\geq \frac1r\,\liminf_{\hat \eps\to 0}\inf_{u\in Y^2}\F_{\hat \eps}(u)\\
	&\geq \theta_0/r
\end{align*}
with $\hat \eps=\eps/r$ as in Lemma \ref{rescaling lemma}. Letting $r\to 0$, we obtain a contradiction.
\end{proof}

If $n=3$, neither the rescaling property of Lemma \ref{rescaling lemma} nor the minimisation property of Lemma \ref{minimisation lemma} hold in as strong formulations. As we sketched in Remark \ref{remark no uniform}, uniform convergence is, in fact, false. 

\begin{proof}[Proof of Theorem \ref{Hausdorff n=2}:]
We assume that $u_\eps\to u$ strongly in $L^1(\Omega)$ and that $\mu_\eps\stackrel*\wto \mu$ for our sequence or a suitable subsequence (not relabelled). Let $D\cc (-1,1)$, then $u_\eps^{-1}(D)\cc \Omega$. We can consider $\overline{u_\eps^{-1}(D)}$ instead without changing the Hausdorff limit to conform with standard approaches. By the usual compactness results (see e.g.\ \cite[Theorem 1.6.6]{krantz2008geometric}), there is a compact set $K$ such that a further subsequence of $u_\eps^{-1}(D)$ converges to $K$ in Hausdorff distance. $K$ can be computed as the Kuratowski lower limit 
\[
K= \{x\in\R^n\:|\:\exists\ x_\eps\in u_\eps^{-1}(D)\text{ such that }x_\eps\to x\}.
\]

{\bf Step 1.} We first show that $K\subset \spt(\mu)$ if $n=2$. Assume $x\in K\setminus\spt(\mu)$. Then there exists $r>0$ such that $B(x,r)\cc \R^2\setminus\spt(\mu)$. This means already that $|u_\eps|\to 1$ uniformly on $B(x,r)$ showing that $u_\eps^{-1}(D)\cap B(x,r) = \emptyset$ for small enough $\eps$, contradicting our assumption.

{\bf Step 2.} Now we prove that $\spt(\mu)\subset K$ in $n=2,3$ dimensions. It suffices to show that for $x\in \spt(\mu), s\in D$ and $r>0$ we have $u_\eps^{-1}(s)\cap B(x,r)\neq \emptyset$ for all sufficiently small $\eps$, which implies Hausdorff convergence. 

This is an easier version of the proof of Lemma \ref{lemma phi sees mu}. Again, we use Corollary \ref{some phi} to show that we have a point $y\in B(x,r/2)$ to use Lemma \ref{lemma local transition} on and then Lemma \ref{lemma local transition} to see that we get points in the pre-image of $s$ close to $y$ since $u_\eps$ is $C^0$-close to an optimal interface in $B(y, 3L\eps)$. 

If $n=2$, the uniqueness of the limit also shows convergence for the whole sequence.
\end{proof}

\mine{
We will now proof our main statement about connectedness.
}

\begin{proof}[Proof of Theorem \ref{theorem ceps}]
The proof is structured as follows. First, we show that we can find neighbourhoods of connected components which have positive distance with respect to the usual metric on $\R^n$. Then we need to show that they also have positive distance with respect to the pseudometric $d^{F(u_\eps)}$. Intuitively, this makes sense since any connecting curve should have to leave the interfacial layer between the two sets. This is simple if $n=2$ and slightly more technical if $n=3$. We proceed in the spirit of Theorem \ref{theorem n=3} and show that $u_\eps$ converges to $\pm 1$ even on curves.

Without loss of generality, we may assume that there are $ -1<\theta_1<\theta_2<1$ such that $\{\phi>0\}\subset (\theta_1,\theta_2)$ and $F\geq 1$ outside $(\theta_1,\theta_2)$. This is only a minor assumption and could easily be removed, but simplifies the proof.

{\bf Step 1.} Assume that $\spt(\mu)$ is not connected. Since $\spt(\mu)$ is compact, according to Lemma \ref{lemma measure support} there are disjoint open sets $U_1, U_2$ such that
\[
\spt(\mu) \subset U_1\cup U_2, \qquad \mu(U_i)>0, i=1,2, \qquad\delta:= \dist(U_1,U_2) >0.
\]
Now
\begin{align*}
\liminf_{\eps\to 0}C_\eps(u_\eps) &\geq \liminf_{\eps\to 0}\int_{U_1}\phi(u_\eps(x))\d x \cdot \liminf_{\eps\to 0} \int_{U_2}\phi(u_\eps(y))\d y\\
	&\qquad\qquad\cdot \liminf_{\eps\to 0}\dist^{F(u_\eps)}(U_1,U_2).
\end{align*}
Since the first two factors are strictly positive according to Lemma \ref{lemma phi sees mu}, it suffices to show that $\liminf_{\eps\to 0}\dist^{F(u_\eps)}(U_1, U_2)>0$.

{\bf Step 2.} For a contradiction, assume that $\dist^{F(u_\eps)}(U_1, U_2) \to 0$. Pick a sequence $c_\eps$ such that $c_\eps \to 0$ but still $\frac{\dist^{F(u_\eps)}(U_1, U_2)}{c_\eps} \to 0$. Then there exist a connected set $K_\eps$ and points $x_\eps, y_\eps\in \overline\Omega$ such that
\[
x_\eps \in K_\eps \cap \partial U_1, \qquad y_\eps \in K_\eps\cap \partial U_2, \qquad \int_{K_\eps}F(u_\eps)\d\H^1 \leq c_\eps.
\]
If $n=2$, we know that $|u_\eps|\to 1$ uniformly on $\Omega\setminus (U_1\cup U_2)$, so in particular $u_\eps\notin [\theta_1,\theta_2]$ on $K_\eps\subset \Omega\setminus (U_1\cup U_2)$ and
\[
\int_{K_\eps} F(u_\eps) \d\H^1 \geq \H^1(K_\eps \setminus (U_1\cup U_2)) \geq \delta >0
\]
since $K_\eps$ connects $U_1$ to $U_2$. This is a contradiction to our assumption. In the case $n=3$ we need a further argument.

{\bf Step 3.} In this step, we will use the competition between the distance function driving $u_\eps$ away from $\pm 1$ along $K_\eps$ and the energy bounds in three dimensions.

Precisely, knowing that the $F(u_\eps)$-weighted length of $K_\eps$ is small, we will construct a set of $N_\eps\gg 1$ points $x_{i, \eps}$ on $K_\eps$ such that $u_\eps(x_{i,\eps}) \in [\theta_1, \theta_2]$ for all $i=1, \dots N$ and such that the balls of radius $\eps^{1/2}$ around these points are disjoint. The large number of balls implies that one of them has to satisfy the conditions of Lemma \ref{minimisation lemma}. This allows us to show that $\alpha$ must have an atom of fixed minimal size close to this point.

A more careful construction allows us to localise the argument on short segments of $K_\eps$ such that we may construct any finite number of atoms with the same fixed minimal size. This gives the desired contradiction.

Take a subsequence realising the $\liminf$. The $1$-Lipschitz map $\pi(x):= \dist(x, U_1)$ maps $K_\eps$ to a connected set containing $0 = \pi(x_\eps)$ and $\delta = \pi(y_\eps)$, so $[0, \delta]\subset \pi(K_\eps)$. Furthermore, $\pi^{-1}(0,\delta)\subset \R^2\setminus (U_1\cup U_2)$ since $\dist(U_1,U_2) = \delta$. Take the set 
\[
K_\eps' := \left\{t\in [0, \delta]\::\:\exists\ x\in K_\eps\text{ such that }t = \pi(x)\text{ and }u_\eps(x)\notin [\theta_1,\theta_2]\right\}
\]
of points whose pre-image contributes a lot to the weighted length of $K_\eps$. Then
\begin{align*}
\H^1(K_\eps') &\leq \H^1\left(K_\eps \cap \{u_\eps\notin [\theta_1,\theta_2]\}\right)\\
	&\leq \int_{K_\eps}F(u_\eps)\d\H^1\\
	&\leq c_\eps.
\end{align*}
Pick $M$ intervals 
\[
I_k = \left[\frac{2k-1}{2M}\,\delta, \:\frac {k}M\,\delta\right]
\]
inside $[0,\delta]$. Fix $1\leq k\leq M$. When $\eps$ is so small that $c_\eps < \frac\delta{4M}$, we deduce that
\begin{equation}
\H^1(I_k\setminus K_\eps') \geq \H^1(I_k) - \H^1(K_\eps') \geq \frac{\delta}{2M} - \frac{\delta}{4M} = \frac\delta{4M}.
\end{equation}
Take a maximal collection $t_{1,\eps},\dots,t_{N_\eps,\eps}$ in $I_k$ with the property that the balls $B(t_{i,\eps}, \eps^{1/2})$ are mutually disjoint. Then $N_\eps \geq \frac{\delta}{16M\eps^{1/2}}-1$ since otherwise
\[
I_k \setminus K_\eps' \subset \bigcup_{i=1}^{N_\eps} B(t_{i,\eps}, 2\eps^{1/2})\qquad\Ra\quad \H^1(I_k\setminus K_\eps') \leq N_\eps\cdot 4\eps^{1/2} \leq \left(\frac{\delta}{16M\eps^{1/2}}-1\right)\cdot 4\eps^{1/2} < \frac{\delta}{4M}. 
\]
For each $t_{i,\eps}$ we pick $x_{i,\eps} \in K_\eps \cap \pi^{-1}(t_{i,\eps})$ which satisfies by construction $u_\eps(x_{i,\eps})\in [\theta_1,\theta_2]$. Since the projection map is contracting, also the balls $B(x_{i,\eps}, \eps^{1/2})$ are mutually disjoint. Fix any $0<\eta<1/2$. Now assume for a contradiction that all points $x_{i,\eps}$ satisfy 
\[
\alpha_\eps(B(x_{i,\eps},\eps^{1/2})) + \int_{B(x_{i,\eps},\eps^{1/2}) \cap\{|u_\eps|>1\}} \frac{W'(u_\eps)^2}{\eps^3}\d x \geq \eps^\eta
\]
for all $i= 1, \dots, N_\eps$. Thus
\[
2\alpha_\eps(\Omega) \geq \alpha_\eps(\Omega) + \int_{\{|u_\eps|>1\}} \frac{W'(u_\eps)^2}{\eps^3}\dx \geq N_\eps \eps^\eta \geq \frac{\delta}{16M } \,\eps^{\eta-1/2} - \eps^\eta \sim \eps^{\eta-1/2}
\]
due to \eqref{eq W'3}. This grows beyond any bound as $\eps\to 0$, so there must by $1\leq i_\eps \leq N_\eps$ such that 
\[
\alpha_\eps(B(x_{i,\eps},\eps^{1/2})) + \int_{B(x_{i,\eps},\eps^{1/2}) \cap\{|u_\eps|>1\}} \frac{W'(u_\eps)^2}{\eps^3}\d x \leq \eps^\eta.
\]
Using compactness and a subsequence, there exists $\overline x^k \in \pi^{-1}(I_k)$ such that $x_{i_\eps, \eps}\to \overline x^k$. We can achieve this simultaneously for all $k=1,\dots,M$. The balls $B(\overline x^k, r)$ are disjoint for $0<r<\frac{\delta}{4M}$ since the intervals $I_k$ have distance $\delta/2M$ and $\mu(B(\bar x^k, r)) = 0$ since 
\[
\dist(\bar x_k, \spt(\mu)) \geq \dist(\overline{\Omega\setminus U_1\cup U_2}, \spt(\mu)) \geq \delta.
\]
Fix such $r$ and consider $\hat \eps = \eps/(2r)$ and
\[
\hat u_\eps:B(0,1)\to \R, \qquad \hat u_\eps (y) = u_\eps\left(\frac{r}2\cdot y + x_{i,\eps}\right)
\]
with $\hat u_\eps(0) \in[\theta_1,\theta_2]$. We use the terminology of Lemma \ref{rescaling lemma}. The associated measures $\hat \alpha_\eps$ satisfy 
\[
\hat \alpha_\eps(B(0,1)) + \int_{B(0,1) \cap\{|\hat u_\eps|>1\}} \frac{W'(\hat u_\eps)^2}{\hat \eps^3}\d x \leq \hat\eps^\eta
\]
so by Lemma \ref{minimisation lemma} 
\begin{align*}
\alpha(B(\bar x^k,r)) &\geq \liminf_{\eps\to 0} \alpha_\eps(B(\bar x^k,3r/4))\\
	&\geq \liminf_{\eps\to 0} \alpha_\eps(B(x_{i_\eps,\eps},r/2))\\
	&=  \liminf_{\eps\to 0} \left[\alpha_\eps(B(x_{i_\eps,\eps},r/2)) + (r/2)^{1-n}\,\mu_\eps(B(x_{i_\eps,\eps}, r/2))\right]\\
	&= \liminf_{\eps\to 0} (\hat\mu_\eps + \hat \alpha_\eps)(B(0,1))\\
	 &\geq \theta_0.
\end{align*}
Taking $r\to 0$, we see that $\alpha(\{\bar x^k\})\geq \theta_0$. So in total, $\alpha(\overline\Omega) \geq M\theta_0$ for all $M\in\N$ where $\theta_0$ depends only on $\theta_1$ and $\theta_2$. Thus we have reached the desired contradiction.
 \end{proof}

\mine{Now Theorem  \ref{main theorem} is an obvious consequence of Theorem \ref{theorem ceps}. }

\begin{proof}[Proof of Theorem \ref{main theorem}:]
Let $u_\eps$ be a sequence such that $\E_\eps(u_\eps)$ is bounded. Then in particular $|\mu_\eps(\Omega)-S|\leq \eps^{\sigma/2}$, so $\mu_\eps(\R^n) = \mu_\eps(\Omega)$ is bounded and $\mu_\eps\wto \mu$ for some Radon measure $\mu$ -- for this and other properties see \cite[Chapter 1]{Evans:1992uy}. Clearly
\[
\mu(\overline\Omega)\geq\limsup_{\eps\to 0}\mu_\eps(\overline\Omega) = S
\]
and on the other hand
\[
\mu(\overline\Omega) \leq \mu(\R^n) \leq \liminf_{\eps\to 0}\mu_\eps(\R^n) = S
\]
so $\mu(\overline\Omega) = S$. If $U =   \R^n\setminus \overline \Omega$, we have
\[
\mu(U) \leq \liminf_{\eps\to 0}\mu_\eps(U) = 0,
\]
so $\spt(\mu) = \bigcap_{U\text{ open}, \mu(U) = 0}U^c\subset\overline\Omega$. \mine{Since $\E_\eps(u_\eps)$ is bounded, we have $C_\eps(u_\eps)\to 0$, so due to Theorem \ref{theorem ceps}, $\spt(\mu)$ is connected.} 
\end{proof}

\mine{
Finally, the same arguments give convergence $\H^1$-a.e.\ on fixed curves.
\begin{proof}[Proof of Theorem \ref{theorem n=3}]
This proof is a simpler version of that of Theorem \ref{theorem ceps}. If $|u_\eps|\not \to 1$ $\H^1$-almost everywhere on $K$, there is a set
\[
K_s':= \{x\in K\:|\:\limsup_{\eps\to 0} \big|\,|u_\eps|-1\,\big|\geq s\}
\]
with $\H^1(K_s')>0$ for some small $s>0$. We do not need the projection onto the first coordinate and can directly define $x_{i,\eps}$ in the fixed set $K_\eps\equiv K_s'$ without passing through $t_{i,\eps}$, the rest works analogously.
\end{proof}
}

\section{Computer Implementation}\label{section computer}

In a computer simulation, we try to find local minimisers of the $\eps$-problem in $n=2$ dimensions\stephan{. In the implementation, we follow a finite element approximation of the} time-normalised $L^2$-gradient flow of $\E_\eps$ with $\eps = 1.5\cdot 10^{-2}$, $\kappa =1$, $\sigma = 2$. The pictures below \stephan{were obtained on a triangular mesh with elements of diameters approximately $h = 6\cdot 10^{-3}$ and approximately $250.000$ $H^2$-conforming subdivision surface basis functions supported in the two-ring around an element. Time-stepping was semi-implicit with a time-step of $\tau = \eps\cdot 10^{-5}$.} The distance function $d^{F(u_\eps)}$ and its gradient are implemented via Dijkstra's algorithm in a fashion similar to the one of \cite{benmansour:2010dm}. The initial condition is the same for all simulations and can be seen in Figure~\ref{fig:nopen} on the left; the domain is the disc of radius $1$. There was no penalisation of the discrepancy measure in our simulations.  

For practical purposes, we use two functions $\phi_1,\phi_2$ with support close to $1$ and $-1$, respectively, rather than just one $\phi$. Quite obviously, our proofs easily extend to that situation. By keeping level sets close to the edges connected, we create barriers that prevent the interface from splitting apart early in the process. The implementation will be described in greater detail in a forth-coming article~\cite{DW_numerical}.

We see in Figure~\ref{fig:nopen} that without the inclusion of the topological term, the transition layer disintegrates into several connected components along the gradient flow of $\W_\eps + \eps^{-\sigma}\,(S_\eps -S)^2$.

\begin{figure}[ht!]\begin{center}
\includegraphics[height =4 cm, clip = true, trim = 110cm 0cm 12cm 0cm]{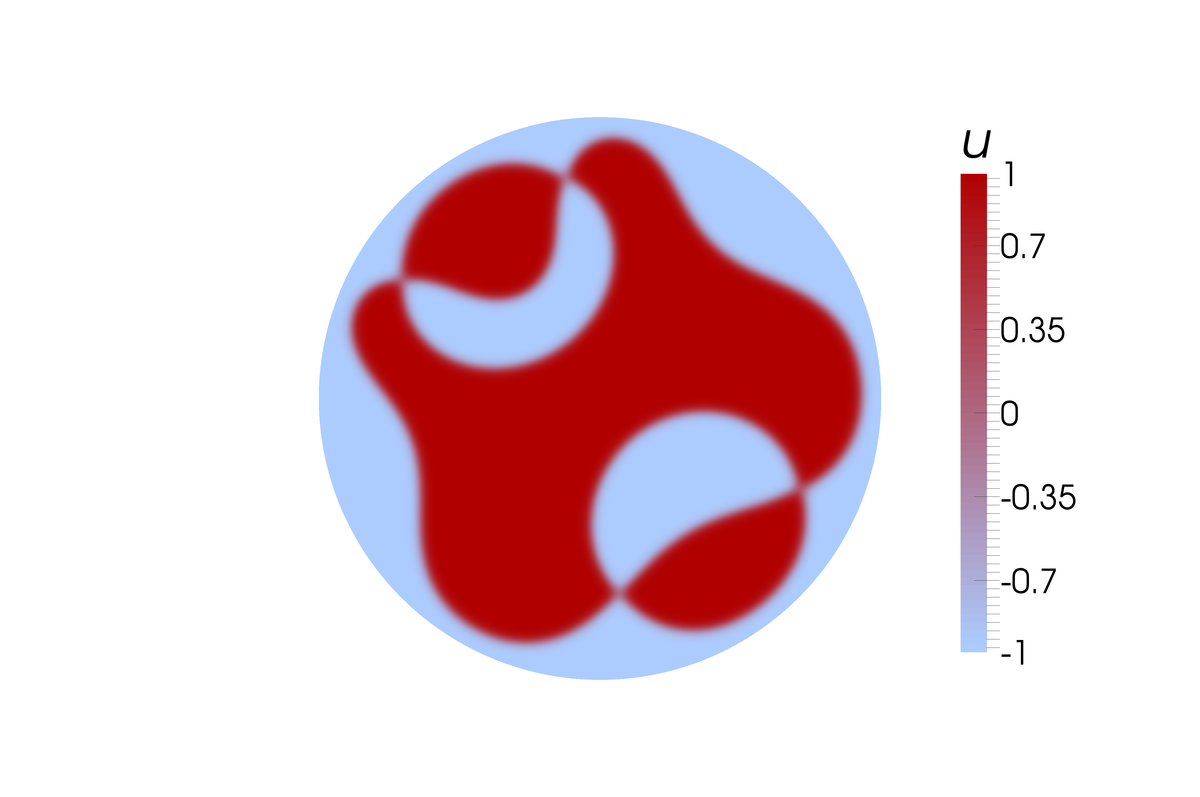}
\includegraphics[height =4 cm, clip = true, trim = 35cm 0cm 35cm 0cm]{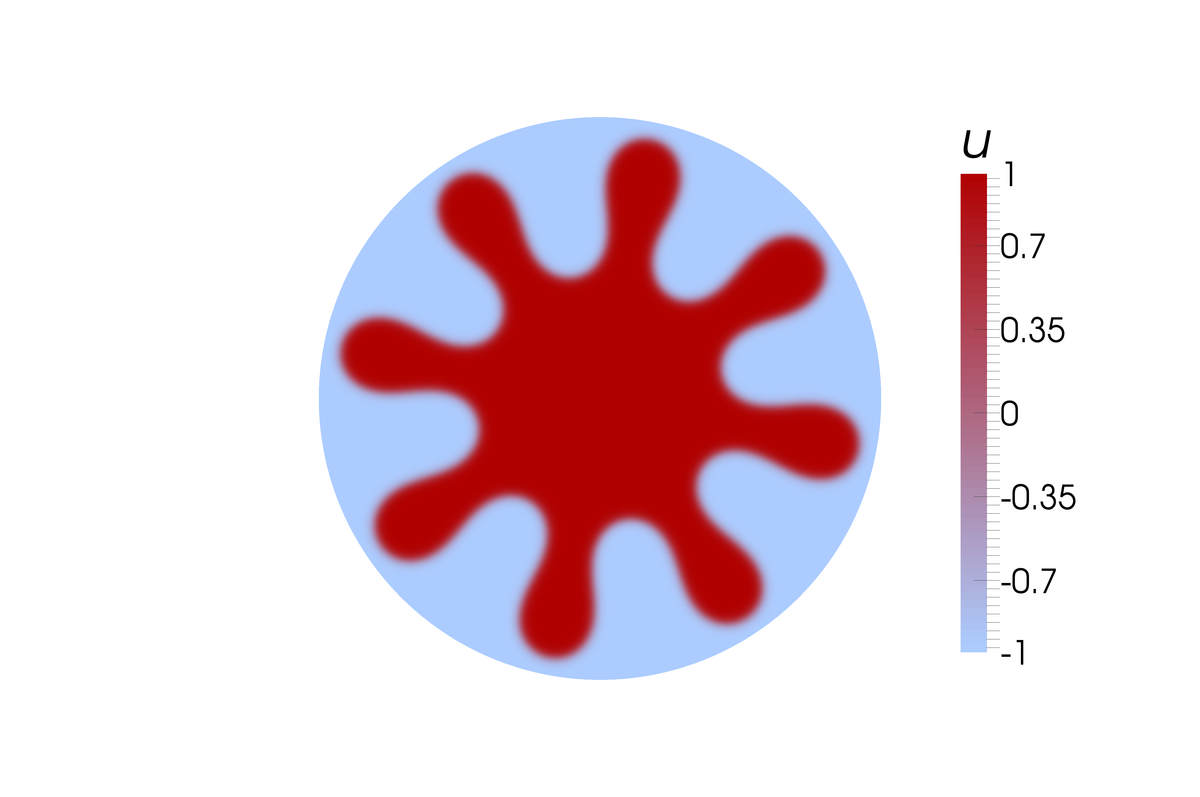}\hspace{0.2cm}
\includegraphics[height =4 cm, clip = true, trim = 35cm 0cm 35cm 0cm]{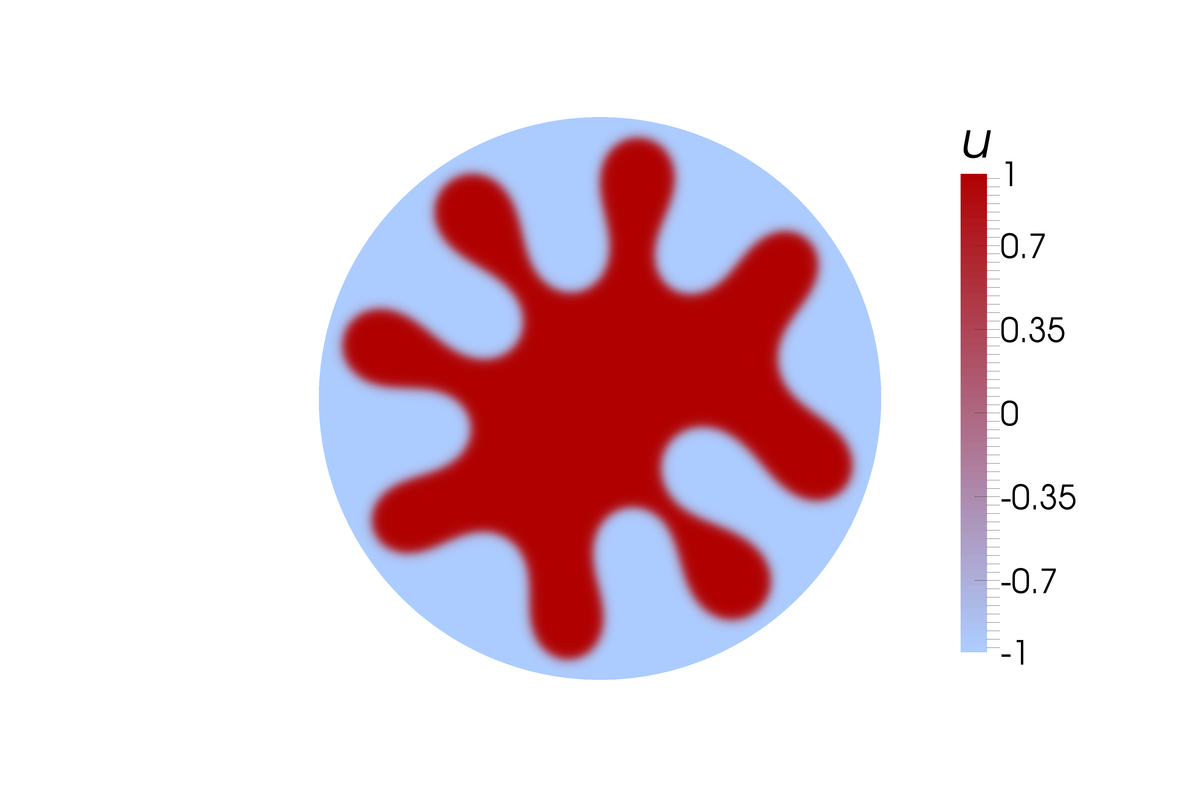}\hspace{0.2cm}
\includegraphics[height =4 cm, clip = true, trim = 35cm 0cm 35cm 0cm]{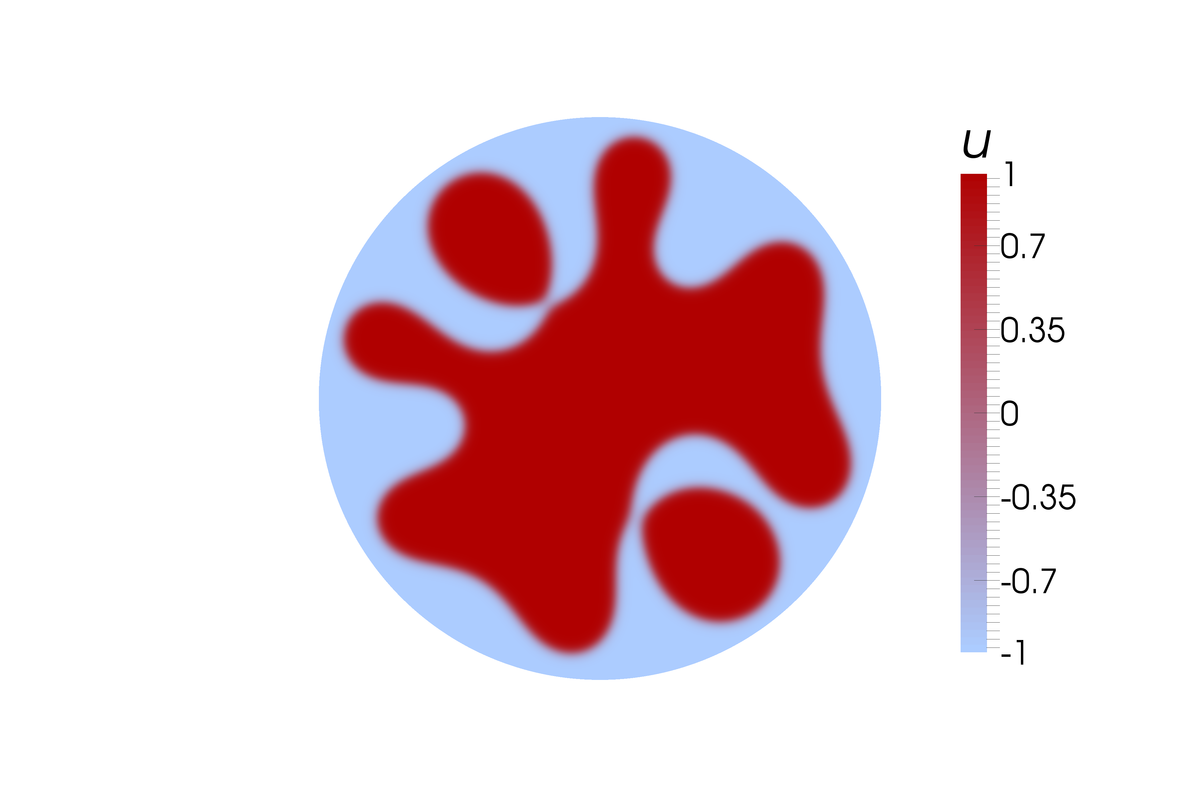}\hspace{0.2cm}
\includegraphics[height =4 cm, clip = true, trim = 35cm 0cm 12cm 0cm]{300w}
\caption{Gradient flow of $\W_\eps + \eps^{-\sigma}\,(S_\eps -S)^2$. From left to right: phase field $u$ for approximately $t=7.5\cdot 10^{-5}$, $t=3\cdot 10^{-4}$, $t=7.5\cdot 10^{-4}$ and $t=1.8\cdot 10^{-3}$.} \label{fig:nopen}
\end{center}
\end{figure}

To compare implementations of topological side conditions, we include the topological term suggested in \cite{Dondl:2011eh}, which penalises a deviation of a diffuse signed curvature integral from $2\pi$ in the simulation. This term prevents the initial pinch-off, but at a later time, the interface will pinch off in a more complicated way which keeps the diffuse winding number close to $2\pi$. The \stephan{phenomenon is a} simultaneous pinch off at several points as seen in Figure~\ref{fig:windingpen}. The far right plot in Figure~\ref{fig:windingpen} illustrates the diffuse curvature density as distributed along the curve at pinch off time. We can observe the formation of a circle with negative total curvature $\approx -2\pi$ (due to the phase field switching in the other direction from $+1$ to $-1$), and two components with total curvature $\approx 2\pi$ so that the total curvature of the whole interface stays close to $2\pi$. 

\begin{figure}[ht!]
\begin{center}
\includegraphics[height =4 cm, clip = true, trim = 110cm 0cm 12cm 0cm]{300w}
\includegraphics[height =4 cm, clip = true, trim = 35cm 0cm 35cm 0cm]{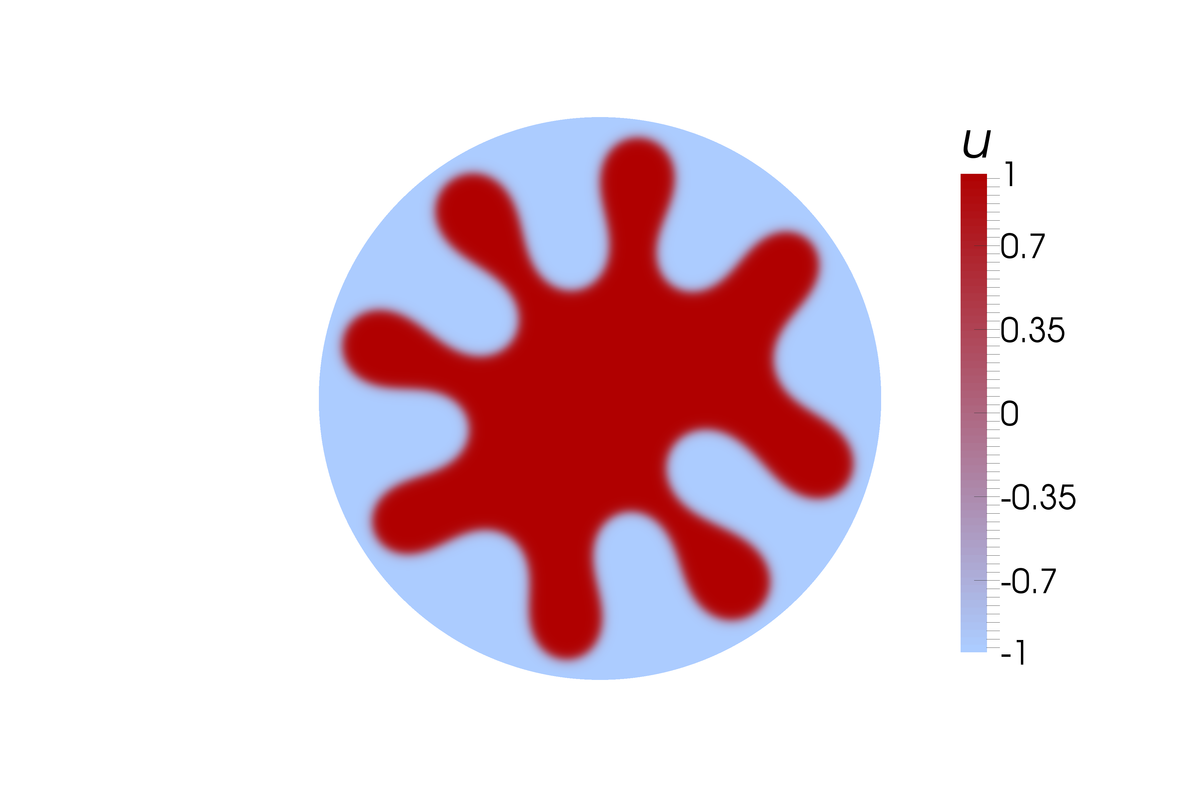}\hspace{0.2cm}
\includegraphics[height =4 cm, clip = true, trim = 35cm 0cm 35cm 0cm]{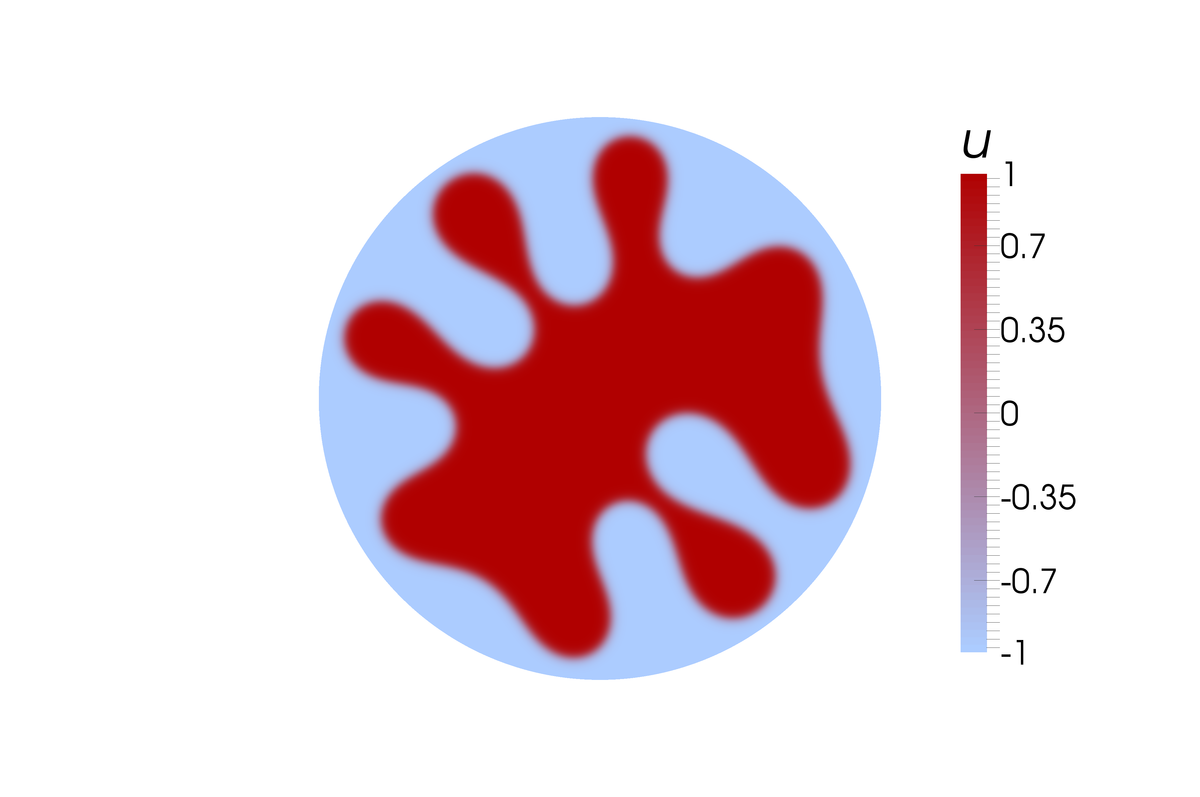}\hspace{0.2cm}
\includegraphics[height =4 cm, clip = true, trim = 35cm 0cm 35cm 0cm]{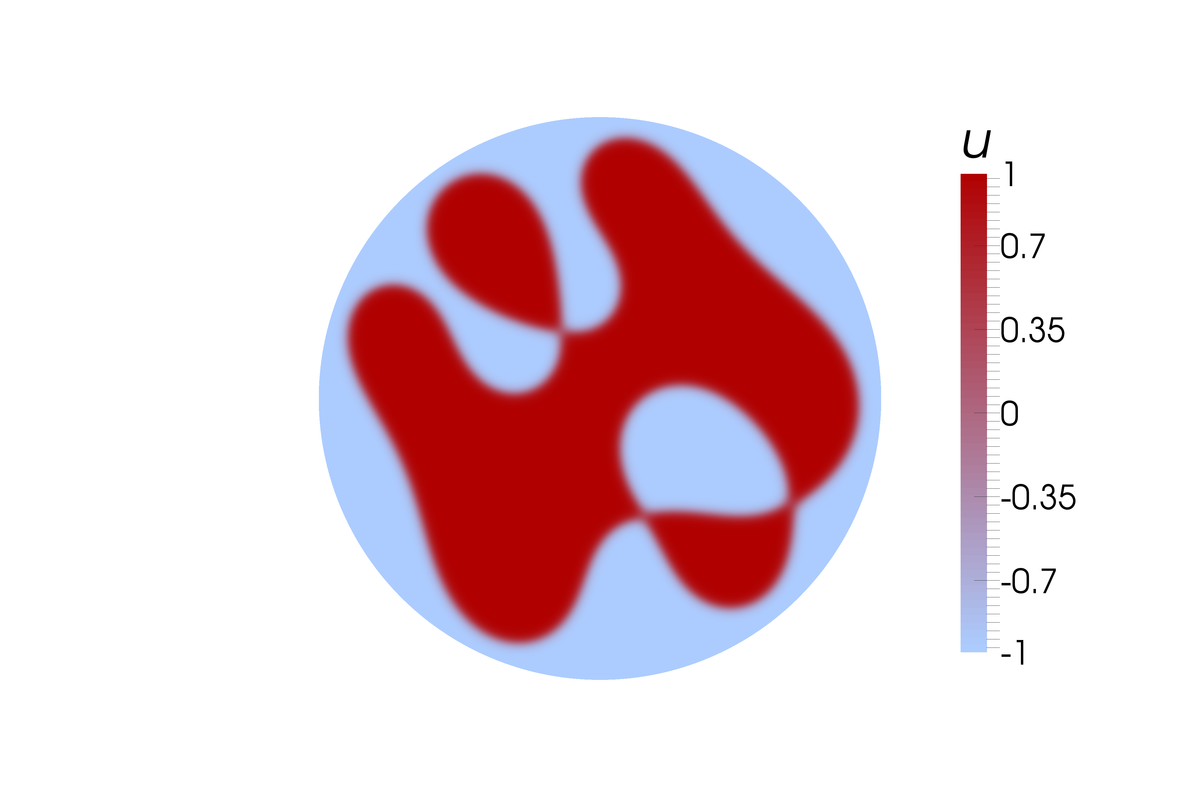}\hspace{0.2cm}
\includegraphics[height =4 cm, clip = true, trim = 35cm 0cm 12cm 0cm]{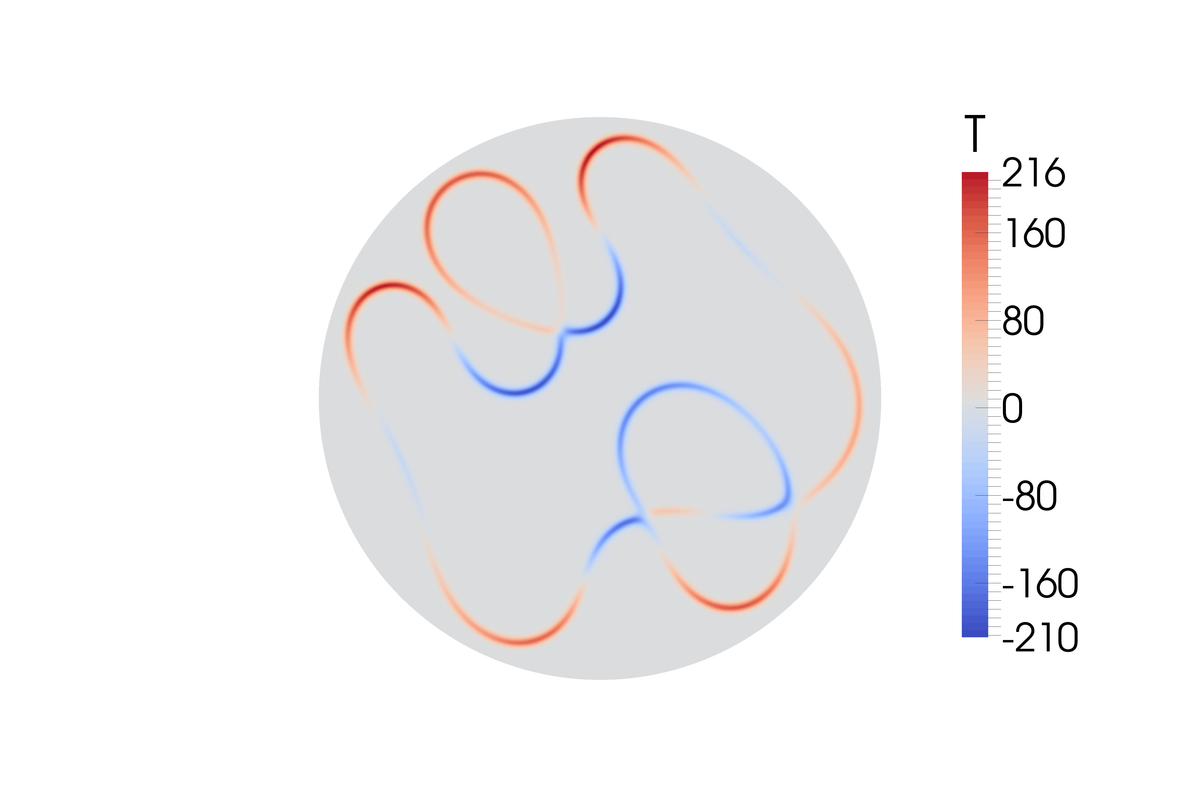}
\caption{Gradient flow with penalty on a diffuse winding number as suggested in~\cite{Dondl:2011eh}. From left to right: phase field $u$ for approximately $t=3\cdot 10^{-4}$, $t=7.5\cdot 10^{-4}$ and $t=1.8\cdot 10^{-3}$, then a plot of the diffuse winding number density denoted $T$ at time $t = 1.8\cdot 10^{-3}$.} \label{fig:windingpen}
\end{center}
\end{figure}

In Figure~\ref{fig:newflow2}, a flow for $\E_\eps$ with the additional term of $C_\eps$ on the other hand can be seen to stably flow past those singular situations.

\begin{figure}[ht!]
\begin{center}
\includegraphics[height =4 cm, clip = true, trim = 110cm 0cm 12cm 0cm]{300w}
\includegraphics[height =4 cm, clip = true, trim = 35cm 0cm 35cm 0cm]{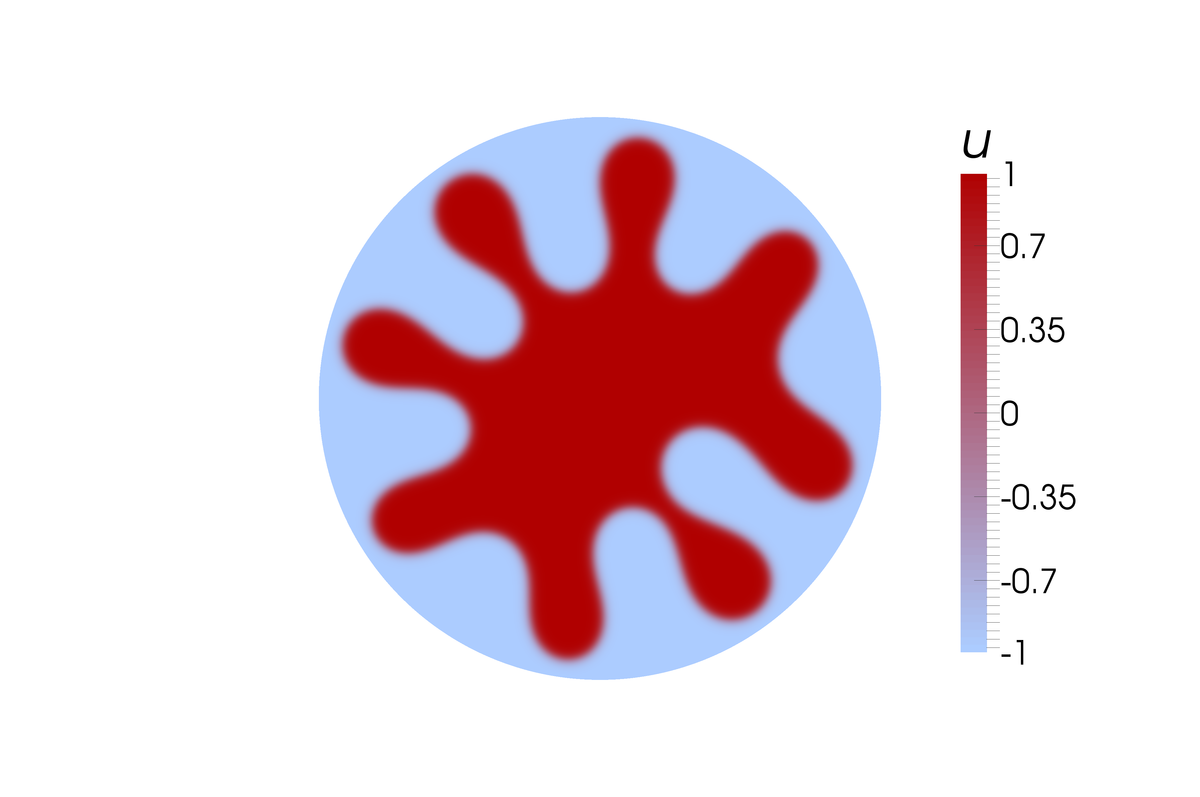}\hspace{0.2cm}
\includegraphics[height =4 cm, clip = true, trim = 35cm 0cm 35cm 0cm]{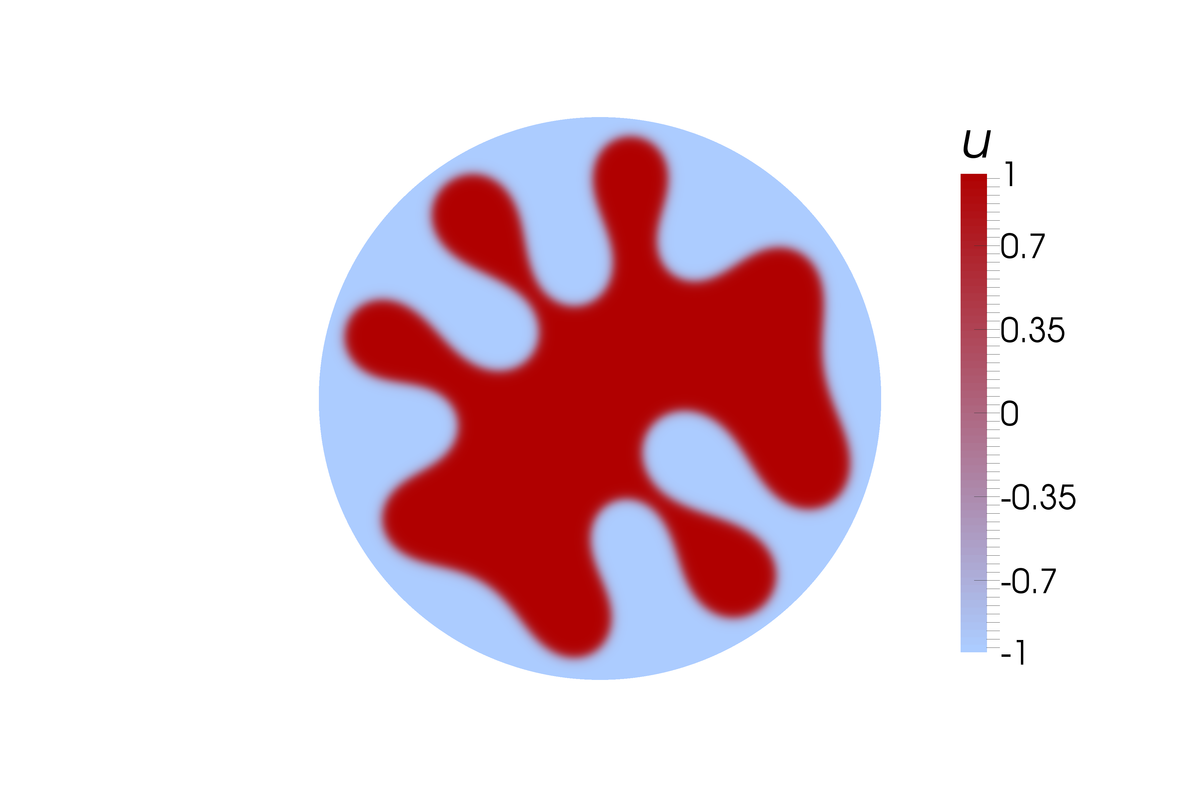}\hspace{0.2cm}
\includegraphics[height =4 cm, clip = true, trim = 35cm 0cm 35cm 0cm]{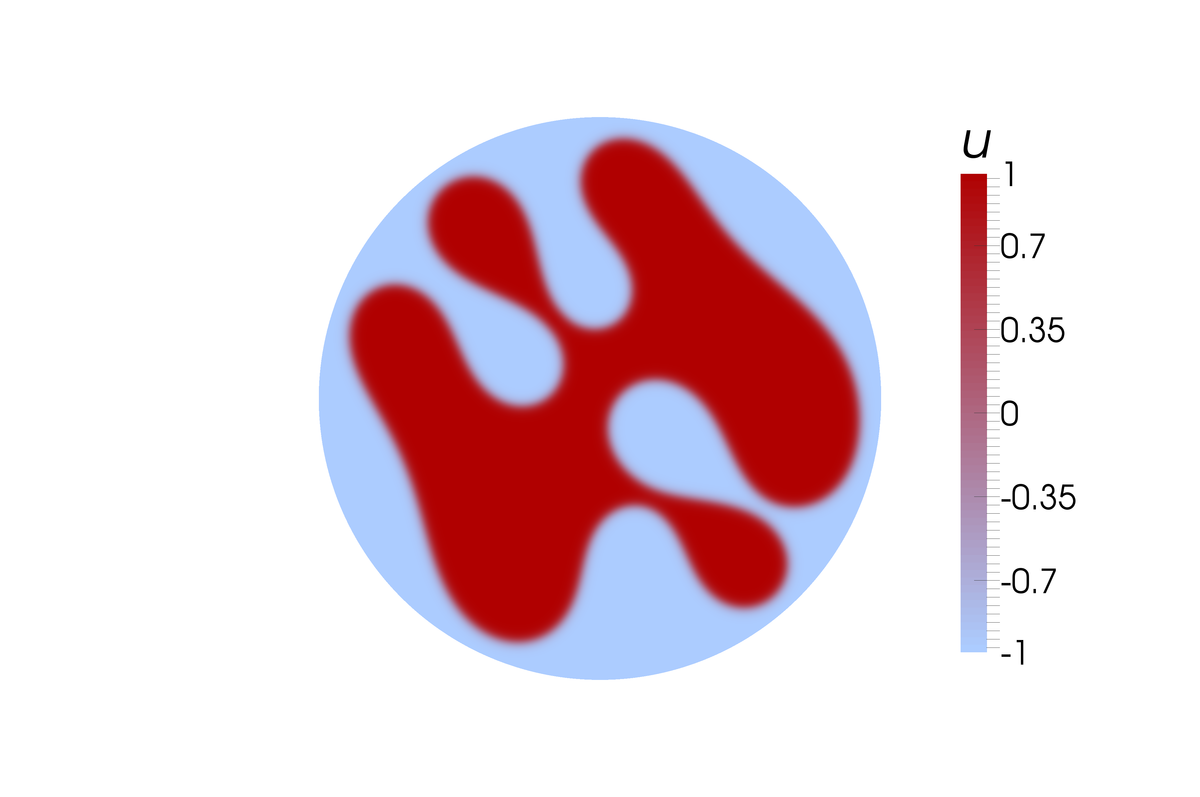}\hspace{0.2cm}
\includegraphics[height =4 cm, clip = true, trim = 35cm 0cm 12cm 0cm]{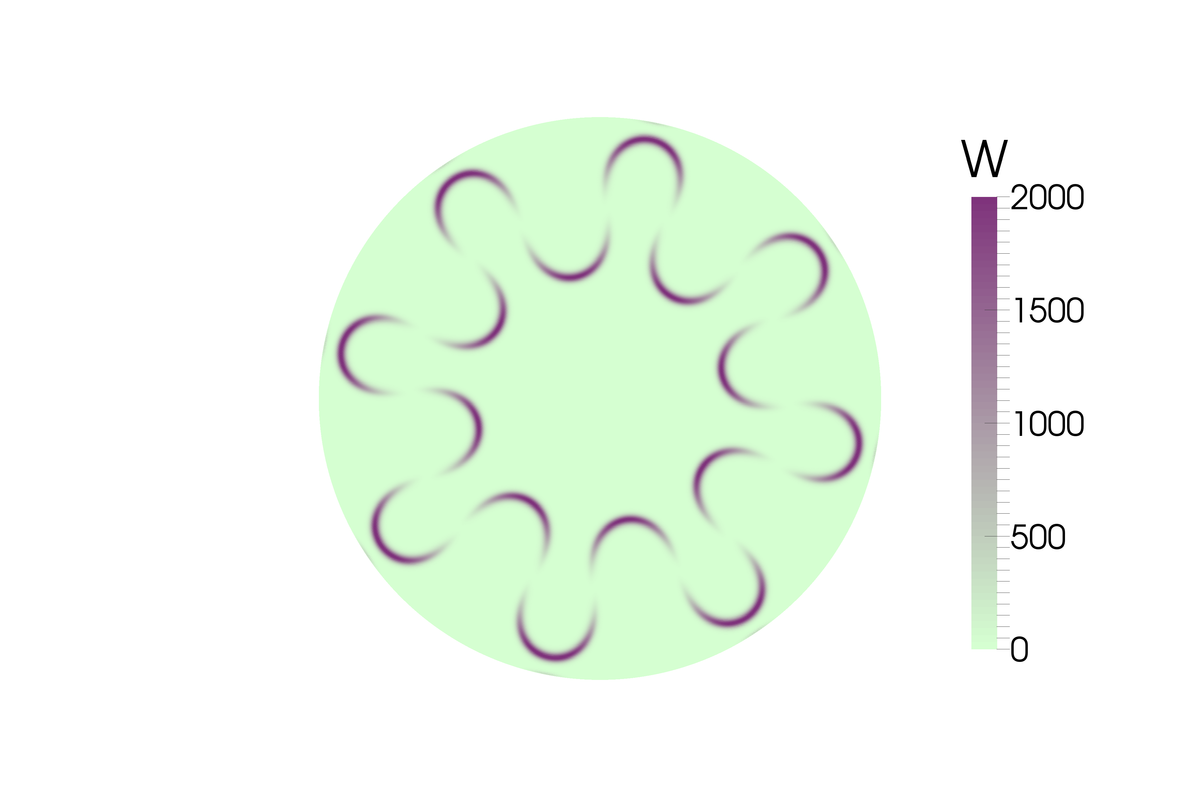}

\includegraphics[height =4 cm, clip = true, trim = 110cm 0cm 12cm 0cm]{300w}
\includegraphics[height =4 cm, clip = true, trim = 35cm 0cm 35cm 0cm]{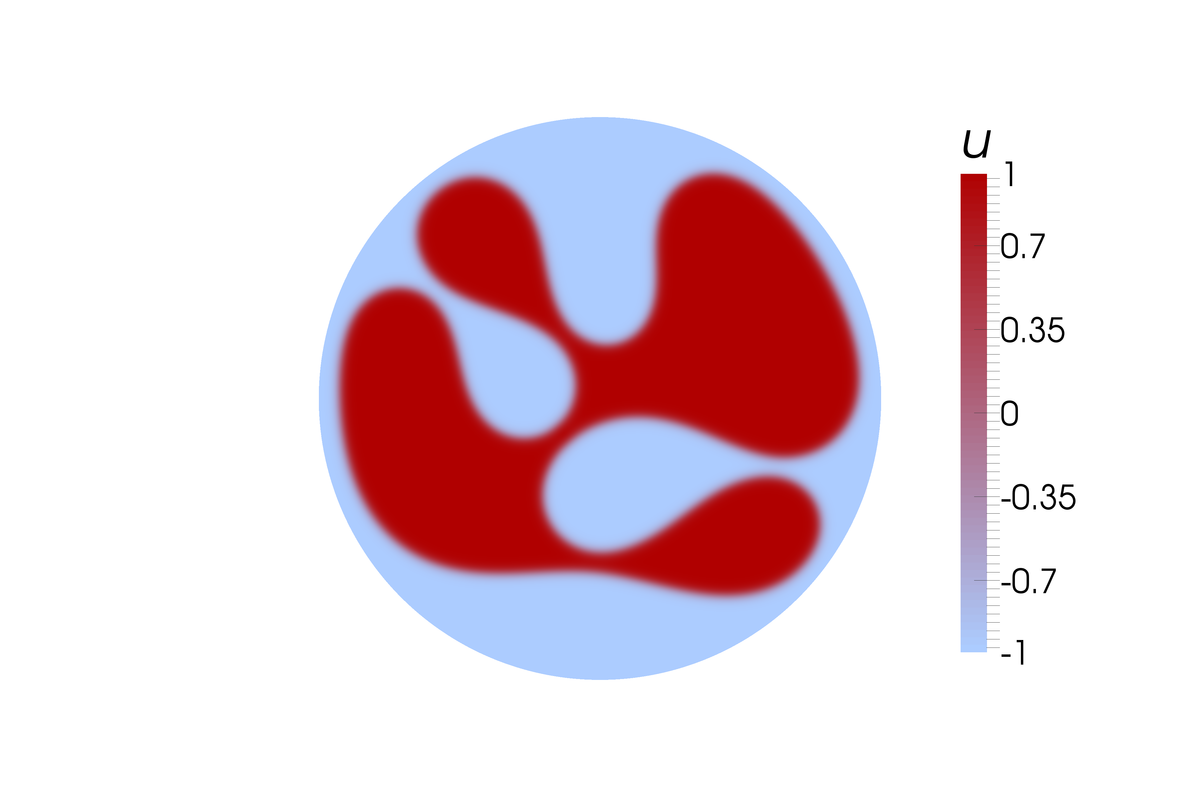}\hspace{0.2cm}
\includegraphics[height =4 cm, clip = true, trim = 35cm 0cm 35cm 0cm]{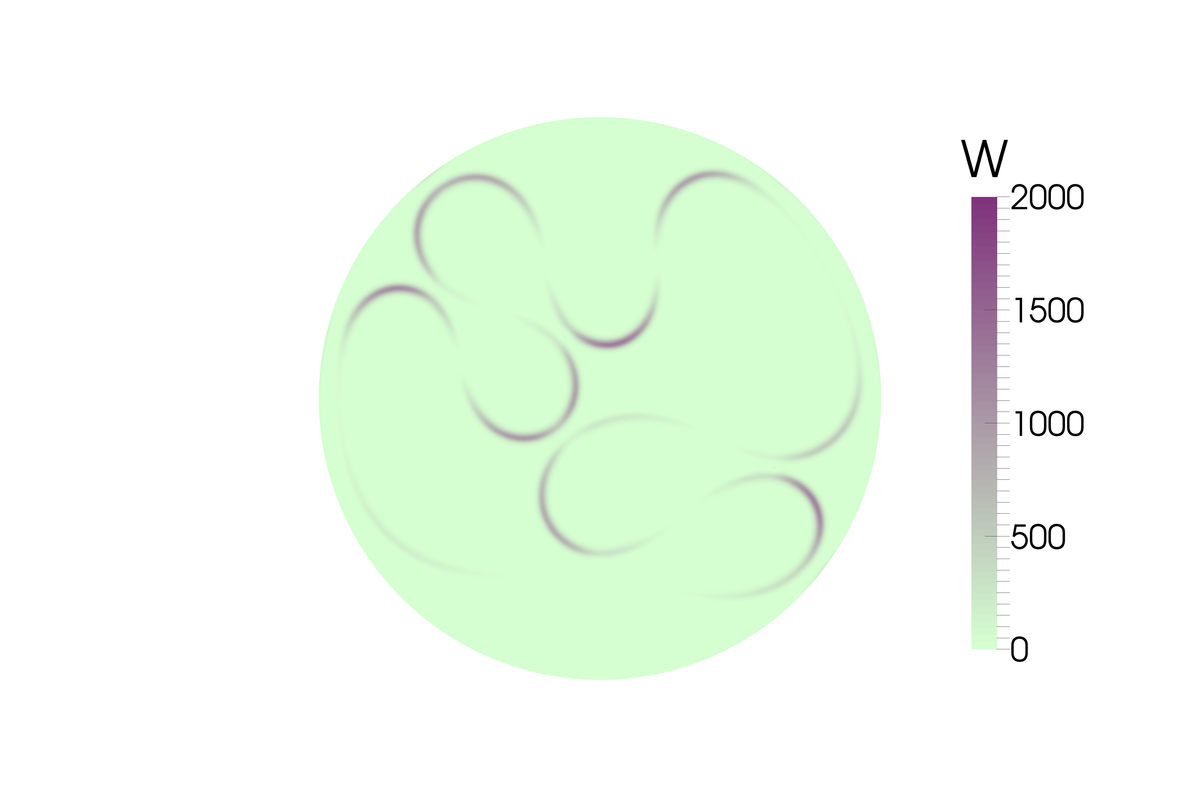}\hspace{0.2cm}
\includegraphics[height =4 cm, clip = true, trim = 35cm 0cm 35cm 0cm]{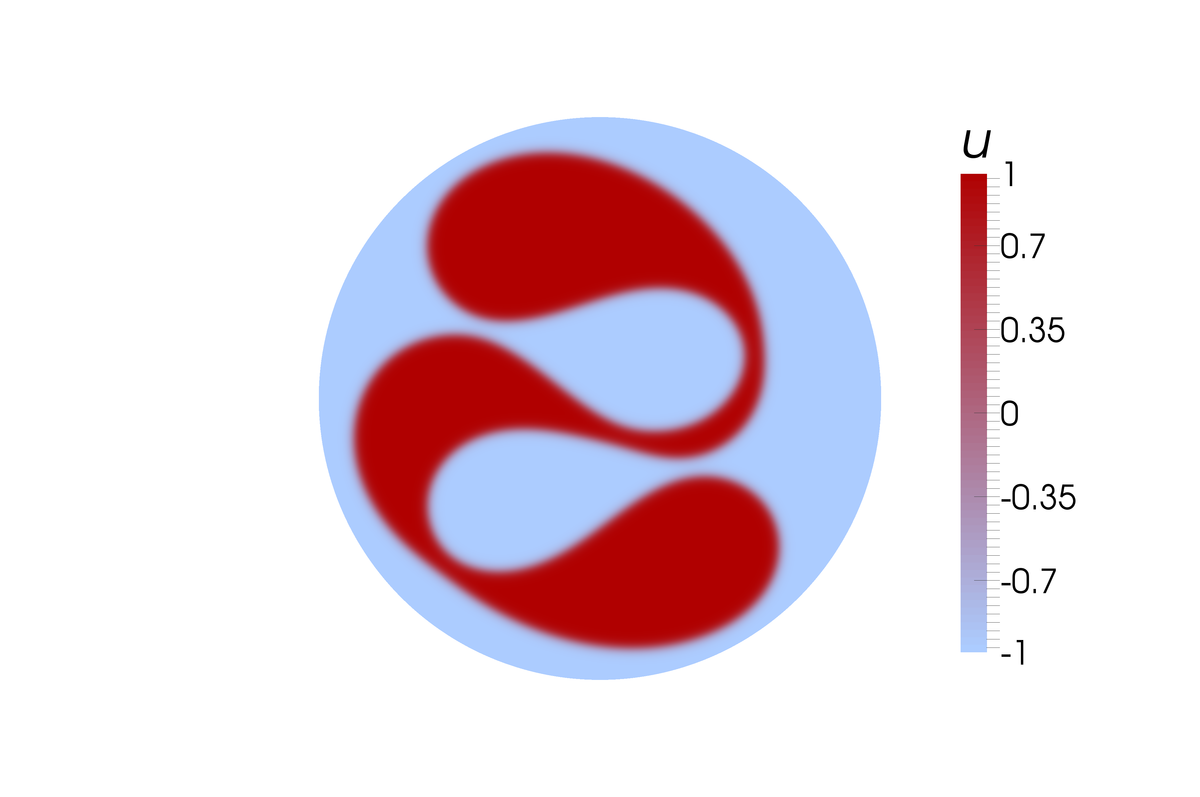}\hspace{0.2cm}
\includegraphics[height =4 cm, clip = true, trim = 35cm 0cm 12cm 0cm]{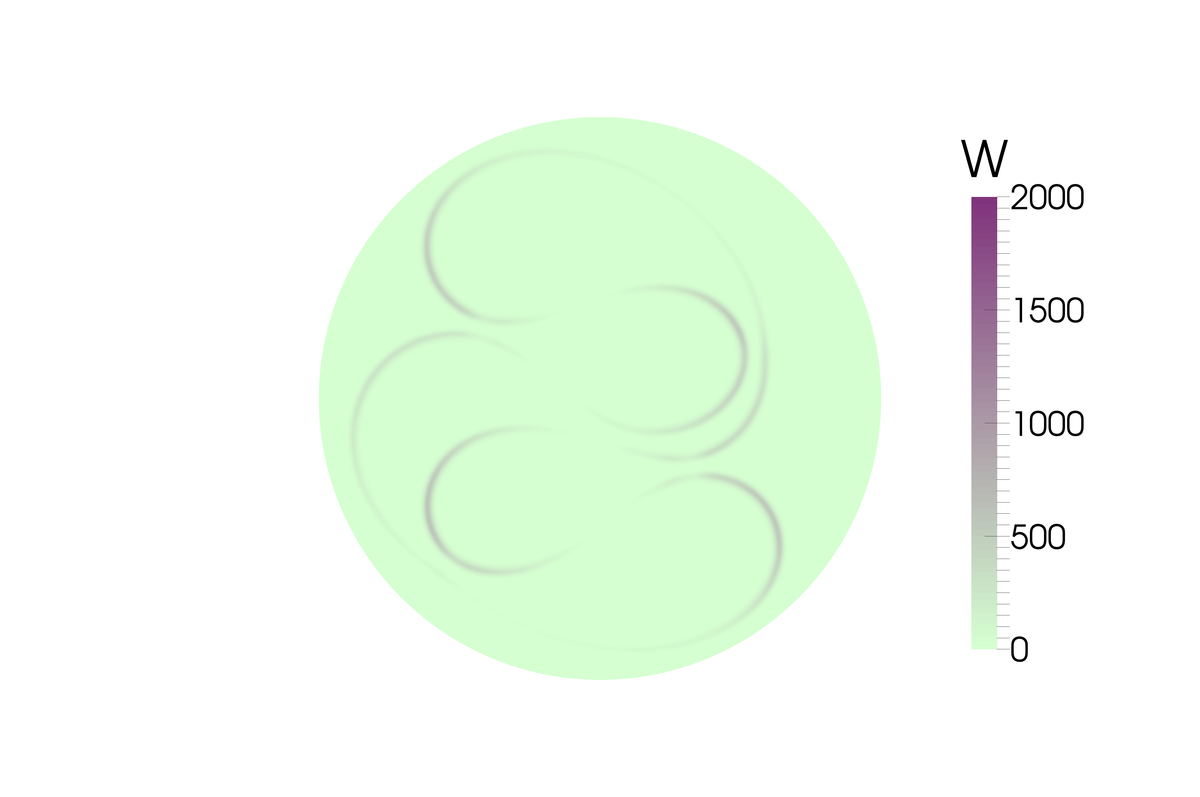}
\caption{\label{fig:newflow2}Evolution including our new topological penalty term $C_\eps$. \stephan{Top line, from} left to right: phase field $u$ for approximately $t=3\cdot 10^{-4}$, $t=7.5\cdot 10^{-4}$ and $t=1.8\cdot 10^{-3}$, then a plot of the diffuse Willmore energy density (denoted $W$ here) of the initial condition.
\stephan{bottom line, left to right}: phase field $u$ and diffuse Willmore energy density first for approximately $t=6.6\cdot 10^{-3}$ and then for approximately $t=3.6\cdot 10^{-2}$.} 

\end{center}
\end{figure}

Comparing the three scenarios above, we observe that there is virtually no difference in the plots at time $3\cdot 10^{-4}$ and that the plots for both modified (penalised using either the old or the new method) functionals at time $7.5\cdot 10^{-4}$ still look very similar. It can thus be argued that the topological condition does not affect the shape of the curve in a major way except when it has to in order to prevent loss of connectedness.

In Figure~\ref{fig:newflow2}, we see non-trivial geometric changes along the gradient flow for later times. This demonstrates the necessity of continuing the flow beyond the critical times.

It should be emphasised that our focus is not on implementing \stephan{a scheme to approximate Willmore flow} using phase fields but on finding minimisers of the diffuse interface problem using a gradient flow. Existence of Willmore flow for long time and topological changes along it are still an open field of research.

\section{Conclusions}\label{section conclusion}

In this paper, we have developed a strategy to enforce connectedness of diffuse interfaces. The strategy fares well in applications and can efficiently be implemented and seems to be more generally applicable to a wider class of problems. We claim that our results can be extended to the following situations.

\begin{itemize}
\item We can include a soft volume constraint like
\[ 
 F\left(  \frac12 \int_\Omega  u_\eps+1\d x \right)
\]
for continuous functions $F\geq 0$.  

\item Another popular constraint compatible with our functional and results is minimising a distance from a given configuration as
\[
A_\eps(u) = \int_\Omega |u - g|\d\lambda
\]
where $\lambda$ is a finite Radon measure on $\Omega$ and $g\in L^1(\Omega)$. This functional originates in problems in image segmentation, but in our context it can be understood as prescribing certain points to lie inside or outside the membrane according to experimental data.
 
\item Using \cite[Theorem 4.1]{Bellettini:2009ui}, we could take Bellettini and Mugnai's approximation of the Helfrich energy 
\[ 
\E_\eps^{\textrm{Hel}}(u) =  \int_\Omega \frac{2+\chi}{2\eps}\, v_{u,\eps}^2 - \frac{\chi}{2\eps}\left|\eps\,\nabla^2u - \frac{W'(u)}\eps \,\nu_u\otimes\nu_u\right|^2\d x   
\]
for $\chi\in(-2,0)$ in place of the diffuse Willmore energy $\W_\eps$. Here $v_{u,\eps}$ is the usual Willmore density associated with $u$ and $\nu_u = \nabla u/|\nabla u|$ is the diffuse normal like in Lemma \ref{proposition varifold convergence}.

\item We can use the same modelling techniques for a finite collection of membranes \stephan{given by $u^1_\eps, \dots, u^N_\eps$} inside an elastic container \stephan{given by $U_\eps$}. \stephan{The governing energy could be composed of a sum of the individual elastic energies $\E_\eps$ and interaction energies $I_\eps$ like
\[
I_\eps (u_\eps^i, u_\eps^j) = \frac1{\eps}\int_\Omega (u_\eps^i+1)^2\,(u_\eps^j+1)^2\dx
\]
which prevent penetration of the phases $u_\eps^i\approx 1$ and $u_\eps^j\approx 1$ or, in a slight variation, enforce confinement of $u_\eps^i\approx 1$ to $U_\eps\approx 1$.}
\end{itemize}

\section*{Acknowledgements}

\noindent PWD acknowledges partial financial support via the ``Wissenschaftler R\"uckkehrprogramm'' of the German Scholars Organization/Carl Zeiss Stiftung. AL has been partially supported by  the ANR research project GEOMETRYA, ANR-12-BS01-0014-01, and  the PGMO research project MACRO ``Mod\`eles d'Approximation Continue de R\'eseaux Optimaux''. SW would like to thank Durham University for financial support through a Durham Doctoral Studentship.

\end{document}